\documentclass[12pt,reqno]{amsart}%
\usepackage{amsmath, amsfonts, amssymb, amsthm, amscd, amsbsy}
\usepackage{fancyhdr}
\usepackage[usenames,dvipsnames,svgnames,x11names,hyperref]{xcolor}
\usepackage{geometry}
\usepackage{graphicx}
\usepackage[pagebackref]{hyperref}%
\usepackage{amsmath}%
\usepackage{amsfonts}%
\usepackage{amssymb}
\usepackage{enumerate}

\usepackage[cmtip,all]{xy}
\usepackage{graphicx}
\usepackage{amssymb}
\usepackage{mathrsfs}
\usepackage{amsfonts}
\usepackage{amsmath}

\providecommand{\U}[1]{\protect\rule{.1in}{.1in}}
\hypersetup{backref=true,
pagebackref=true,
hyperindex=true,
colorlinks=true,
breaklinks=true,
urlcolor=NavyBlue,
linkcolor=Fuchsia,
bookmarks=true,
bookmarksopen=false,
filecolor=black,
citecolor=ForestGreen,
linkbordercolor=red
}
\newtheorem{theorem}{Theorem}[section]
\newtheorem*{theorem*}{Theorem}
\theoremstyle{plain}

\newtheorem{lemma}{Lemma}[section]

\newtheorem{remark}{Remark}[section]

\numberwithin{equation}{section}

\allowdisplaybreaks

\begin{document}

\title[Sharp trace inequalities for
fractional powers of the sublaplacian]{Sharp  trace inequalities for conformally invariant fractional powers of the sublaplacian on the Heisenberg group and the CR sphere}


\author{Qiaohua  Yang}
\address{School of Mathematics and Statistics, Wuhan University, Wuhan, 430072, People's Republic of China}

\email{qhyang.math@whu.edu.cn}

\author{Leyuan Yu}

\address{School of Mathematics and Statistics, Wuhan University, Wuhan, 430072, People's Republic of China}

\email{2023302011095@whu.edu.cn}

\thanks{The first  author  was partially supported by the National Natural Science Foundation of China(No.12471056).}






\begin{abstract}
We establish sharp Sobolev trace inequalities for conformally invariant fractional powers of the sublaplacian on the Heisenberg group and the CR sphere, extending the corresponding Euclidean results of Einav-Loss, Beckner, and Bez-Machihara-Sugimoto to these non-Euclidean settings.
In the limiting case, sharp trace Beckner-Onofri inequalities are also established on the CR sphere. The proofs are based on a duality argument due to Bez-Machihara-Sugimoto, together with the Frank-Lieb sharp form of the Hardy-Littlewood-Sobolev inequalities on the Heisenberg group and the CR sphere.
The same approach also yields trace Beckner-Onofri inequalities on the standard sphere.

  \end{abstract}

\maketitle


\section{Introduction}

Let \( n \geq 3 \) and denote by \(\mathbb{R}^{n}_{+} = \{ (x, t) \in \mathbb{R}^n : x \in \mathbb{R}^{n-1},  t > 0 \}\) the upper half-space in \(\mathbb{R}^n\).
The Sobolev trace inequality on \(\mathbb{R}^{n}_{+}\), first established by Escobar \cite{es} (see also Carlen--Loss \cite{Carlen}), states that for any real-valued function \(f\) on \(\mathbb{R}^{n}_{+}\), sufficiently smooth up to the boundary and decaying rapidly enough at infinity, the following holds:
\begin{align}\label{1.1}
\left( \int_{\mathbb{R}^{n-1}} |(\tau f)(x)|^{\frac{2(n-1)}{n-2}} dx\right)^{\frac{n-2}{n-1}}
\leq C_n \int_{\mathbb{R}^{n}_{+}} |\nabla f(x,t)|^2 dx dt,
\end{align}
where
\[
C_n = \frac{1}{\sqrt{\pi}(n-2)}
\Bigl( \frac{\Gamma(n-1)}{\Gamma\bigl(\frac{n-1}{2}\bigr)} \Bigr)^{\frac{1}{n-1}},
\]
and \((\tau f)(x)\) denotes the trace of \(f\) on \(\partial\mathbb{R}^{n}_{+}\).
The equality in \eqref{1.1} holds if and only if \(f\) is a multiple of
\[
\Bigl( \frac{1}{(\varepsilon + t)^2 + |x - x_0|^2} \Bigr)^{(n-2)/2}, \qquad
\varepsilon > 0, \; x_0 \in \mathbb{R}^{n-1}.
\]

We remark that, via the seminal extension problem of Caffarelli and Silvestre \cite{ca}, sharp Sobolev trace inequalities on \(\mathbb{R}^{n+1}_{+}\) can be established for \(\gamma \in (0,1)\) (see Carlen--Loss \cite{Carlen} for \(\gamma = \frac{1}{2}\)).
In \cite{Case2}, Case extended the framework of Caffarelli and Silvestre by introducing a family of higher-order boundary operators, yielding a corresponding family of sharp Sobolev trace inequalities on \(\mathbb{R}^{n+1}_{+}\).
Building on the work of Case et al. (who generalized these results to a conformally invariant class of manifolds and introduced the relevant conformally covariant boundary operators for the Paneitz and sixth-order GJMS cases \cite{Case0,Case1,Case&Luo}),   the extension to Poincar\'e-Einstein manifolds was later established by Flynn, Lu, and the second author \cite{Flynn&Lu&Yang}.
Analogous results in the CR setting have been developed by Frank et al. \cite{fr1} and by Flynn, Lu, and the second author \cite{Flynn&Lu&Yang1}.
For more information on Sobolev trace inequalities, we refer to \cite{Ache&Chang,be2,Case&Chang,Chen&Zhang,Gong,yang} and the references therein.

In their work \cite{ei}, Einav and Loss investigated sharp trace inequalities for fractional Sobolev spaces. We briefly summarize the  results as follows.
Define the space \( D_s(\mathbb{R}^n) \) to be the space of all functions $f \in L^2(\mathbb{R}^n)$ whose Fourier transform $\widehat{f}$ satisfies
\[
\int_{\mathbb{R}^n} |\widehat{f}(y)|^2  |y|^{s}dy < \infty .
\]
For \( f \in D_s(\mathbb{R}^n) \), let $$ \mathbb{R}^{n-m} = \{(x_1, \dots, x_{n-m}, 0, 0, \dots, 0) : x_j \in \mathbb{R}, 1 \leq j \leq n-m\} $$
 and let \( \tau_m \) denote the restriction operator that maps functions on \(\mathbb{R}^n\) to \(\mathbb{R}^{n-m}\).
Einav and Loss established the following sharp Sobolev trace inequality for fractional Sobolev spaces.
\begin{theorem}(\cite[Theorem 1.1]{ei}) \label{th1}
Let $0 \leq m < n$ and $\frac{m}{2} < s < \frac{n}{2}$. For any $f \in D_{2s}(\mathbb{R}^{n})$ we have
\begin{equation}
\label{eq:trace-ineq}
\|\tau_m f\|_{L^{\frac{2(n-m)}{n-2s}}}^2 \leq C_{m,s,n} \int_{\mathbb{R}^{n}}|(-\Delta)^{s/2}f|^{2},
\end{equation}
where
\[
C_{m,s,n} = 2^{-2s} \pi^{-s} \cdot \frac{\Gamma(n/2 - s) \Gamma(s - m/2)}{\Gamma(s) \Gamma(n/2 + s - m)} \left\{ \frac{\Gamma(n - m)}{\Gamma((n - m)/2)} \right\}^{(2s - m)/(n - m)}.
\]
There is equality only if $f(x)$ is proportional to
\begin{equation*}
\int_{\mathbb{R}^m} \frac{1}{\left(|x'|^2 + |x'' - y''|^2\right)^{(n-2s)/2}} \frac{1}{\left(\gamma^2 + |y'' - a|^2\right)^{(n+2s-2m)/2}} dy''
\end{equation*}
for some $a \in \mathbb{R}^{n-m}$ and $\gamma \neq 0$.
\end{theorem}
In particular,  when \( s = 2 \) and \( m = 1 \), the inequality \eqref{eq:trace-ineq} recovers the earlier result of Escobar.

Moreover,   Beckner \cite{be2} and Bez-Machihara-Sugimoto \cite{bez} established another type of Sobolev trace inequality for fractional Sobolev spaces.
Unlike Beckner's approach, Bez--Machihara-Sugimoto employed a dual  argument. Specifically, they proved the following theorem:

\begin{theorem}(\cite[Theorem 5]{be2}, \cite[Theorem 4.1]{bez}) \label{th2}
Let $n \geq 2$ and  $s \in \left(\frac{1}{2}, \frac{n}{2}\right)$. For $f \in D_{2s}(\mathbb{R}^n)$,
\begin{equation}
\label{eq:thm4.1}
\|\mathcal{R}f\|_{L^{\frac{2(n-1)}{n-2s}}(\mathbb{S}^{n-1})}^2 \leq 2^{1-2s} \frac{\Gamma(2s-1)\Gamma\left(\frac{n}{2}-s\right)}{\Gamma(s)^2\Gamma\left(\frac{n}{2}-1+s\right)} \left( \frac{\Gamma\left(\frac{n}{2}\right)}{2\pi^{\frac{n}{2}}} \right)^{\frac{2s-1}{n-1}} \int_{\mathbb{R}^{n}}|(-\Delta)^{s/2}f|^{2}.
\end{equation}
Here $\mathcal{R}$ denotes  the operator which restricts  $f$ on  $\mathbb{R}^n$ to the unit sphere $\mathbb{S}^{n-1}\subset \mathbb{R}^n$.
\end{theorem}

By stereographic projection, inequality \eqref{eq:thm4.1} is in fact equivalent to the case \( m = 1 \) of inequality \eqref{eq:trace-ineq}. This equivalence can be observed from our subsequent proof in the case of  Heisenberg group.

The purpose of this paper is to extend Theorems \ref{th1} and \ref{th2} to the settings of the Heisenberg group and the CR sphere.
To state our main results, we first introduce some necessary notations.
Recall that the Heisenberg group \(\mathbb{H}^n\) is \(\mathbb{C}^n \times \mathbb{R}\) with elements \(u = (z,t)\) and group law
\[
uu' = (z,t)(w,\mu) = \bigl(z + w, t + \mu + 2\operatorname{Im}(z \cdot \overline{w})\bigr),
\]
where
\[
z \cdot \overline{w} = \sum_{j=1}^{n} z_j \overline{w}_j.
\]
Here \(\operatorname{Im} z\) and \(\operatorname{Re} z\) denote, respectively, the imaginary and real parts of \(z\).
The Haar measure on \(\mathbb{H}^n\) is the usual Lebesgue measure \(du = dzdt\).
We denote the convolution on the Heisenberg group by
\[
f * g(u) = \int_{\mathbb{H}^n} f(v)g(v^{-1}u) dv.
\]
The homogeneous norm on \(\mathbb{H}^n\) is given by
\[
|u| = |(z,t)| = \bigl(|z|^4 + t^2\bigr)^{1/4},
\]
and dilations of a point \(u = (z,t)\) are defined as
\[
\delta_r(u) = (r z, r^2 t),\qquad r > 0.
\]
 Denote by $Q_n$ the homogeneous dimension of \(\mathbb{H}^n\).  Then
$$Q_n = 2n+2.$$
The sublaplacian on \(\mathbb{H}^n\) is the second-order differential operator
\[
\mathcal{L}_0 = -\frac{1}{4} \sum_{j=1}^n \bigl(X_j^2 + Y_j^2\bigr),
\]
where
\[
X_j = \frac{\partial}{\partial x_j} + 2y_j \frac{\partial}{\partial t},\qquad
Y_j = \frac{\partial}{\partial y_j} - 2x_j \frac{\partial}{\partial t},\quad j=1,\cdots,n.
\]

The conformally invariant fractional powers of the sublaplacian on \(\mathbb{H}^{n}\) are defined by (see \cite{Co,ron1})
\[
\mathcal{L}_{s,\mathbb{H}^{n}} \;:=\; |2T|^{s/2}
\frac{\Gamma\left(\mathcal{L}_0 |2T|^{-1} + \frac{2+s}{4}\right)}
{\Gamma\left(\mathcal{L}_0 |2T|^{-1} + \frac{2-s}{4}\right)},
\]
where \(T = \frac{\partial}{\partial t}\).
The operator \(\mathcal{L}_{s,\mathbb{H}^{n}}\) is left-invariant and satisfies the homogeneity property
\begin{align}\label{1.4}
\mathcal{L}_{s,\mathbb{H}^{n}}\bigl(f \circ \delta_r\bigr)
= r^{s}\bigl(\mathcal{L}_{s,\mathbb{H}^{n}} f\bigr)\circ \delta_r,\qquad
f \in C_{0}^{\infty}(\mathbb{H}^{n}).
\end{align}
The Sobolev space \(W^{s,2}(\mathbb{H}^{n})\) is defined as the completion of \(C_{0}^{\infty}(\mathbb{H}^{n})\) with respect to the norm
\[
\|f\|_{W^{s,2}(\mathbb{H}^{n})} \;:=\;
\Bigl(\int_{\mathbb{H}^{n}} \bigl|\mathcal{L}_{s,\mathbb{H}^{n}}^{1/2}f\bigr|^{2}\Bigr)^{1/2}.
\]

Let
\[
\mathbb{S}^{2n+1} = \Bigl\{ z=(\zeta_{1},\dots,\zeta_{n+1}) \in \mathbb{C}^{n+1} \;:\; \sum_{j=1}^{n+1} |\zeta_{j}|^{2}=1 \Bigr\}
\]
be the CR sphere. The sublaplacian on \(\mathbb{S}^{2n+1}\) is defined as
\[
\mathcal{L} = -\frac{1}{2} \sum_{j=1}^{n+1} \bigl( T_j \overline{T}_j + \overline{T}_j T_j \bigr),
\]
where
\[
T_j = \frac{\partial}{\partial \zeta_j} - \overline{\zeta}_j \sum_{k=1}^{n+1} \zeta_k \frac{\partial}{\partial \zeta_k}, \quad j=1,\cdots,n+1.
\]
The transversal direction is given by the real vector field
\[
\mathcal{T} = \frac{i}{2} \sum_{j=1}^{n+1} \Bigl( \zeta_j \frac{\partial}{\partial \zeta_j} - \overline{\zeta}_j \frac{\partial}{\partial \overline{\zeta}_j} \Bigr).
\]
The conformal sublaplacian on \(\mathbb{S}^{2n+1}\) is defined as
\[
\mathcal{D} = \mathcal{L} + \frac{n^2}{4}.
\]
The Sobolev space \(W^{s,2}(\mathbb{S}^{2n+1})\) is defined as the completion of \(C^\infty(\mathbb{S}^{2n+1})\) with respect to the norm
\[
\|F\|_{W^{s,2}(\mathbb{S}^{2n+1})} = \|\mathcal{D}^{s/2}F\|_{L^2(\mathbb{S}^{2n+1})}.
\]

It is well known that \(L^2(\mathbb{S}^{2n+1})\) admits a decomposition into \(U(n+1)\)-irreducible components
\[
L^2(\mathbb{S}^{2n+1}) = \bigoplus_{j,k \geq 0} \mathcal{H}_{j,k},
\]
where \(\mathcal{H}_{j,k}\) consists of the restrictions to \(\mathbb{S}^{2n+1}\) of harmonic polynomials \(p(z,\overline{z})\) on \(\mathbb{C}^{n+1}\) that are homogeneous of degree \(j\) in \(z\) and degree \(k\) in \(\overline{z}\) (see \cite{fo,tha2} and references therein).
We define the operator \(\mathcal{A}_{s,\mathbb{S}^{2n+1}}\) on \(W^{s,2}(\mathbb{S}^{2n+1})\) for \(-Q_n < s \leq Q_n\) by its action on the spherical harmonics \(Y_{jk}\in\mathcal{H}_{j,k}\):
\begin{align}\label{1.5}
\mathcal{A}_{s,\mathbb{S}^{2n+1}} Y_{jk} = \lambda_j(s)\lambda_k(s)Y_{jk},
\end{align}
where
\[
\lambda_j(s) = \frac{\Gamma\bigl(\frac{Q_n+s}{4} + j\bigr)}{\Gamma\bigl(\frac{Q_n-s}{4} + j\bigr)}.
\]
In terms of the conformal sublaplacian \(\mathcal{D}\) and the transversal vector field \(\mathcal{T}\), the operator \(\mathcal{A}_{s,\mathbb{S}^{2n+1}}\) can be expressed as (see Section 2.2; for the case where \(s\) is an even integer, see   \cite{br1})
\begin{align}\label{As}
\mathcal{A}_{s,\mathbb{S}^{2n+1}}
= \frac{\Gamma\Bigl(\sqrt{\mathcal{D}-\mathcal{T}^{2}}-i\mathcal{T}+\frac{2+s}{4}\Bigr)
      \Gamma\Bigl(\sqrt{\mathcal{D}-\mathcal{T}^{2}}+i\mathcal{T}+\frac{2+s}{4}\Bigr)}
     {\Gamma\Bigl(\sqrt{\mathcal{D}-\mathcal{T}^{2}}-i\mathcal{T}+\frac{2-s}{4}\Bigr)
      \Gamma\Bigl(\sqrt{\mathcal{D}-\mathcal{T}^{2}}+i\mathcal{T}+\frac{2-s}{4}\Bigr)}.
\end{align}

Set
\[
\mathbb{H}^{n-m}=\{(z,t)\in \mathbb{H}^{n} : z=(z_{1},\dots, z_{n-m},0,\dots,0) \},
\]
considered as a subgroup of \(\mathbb{H}^n\) via the natural embedding.
For convenience, we write a point $z\in\mathbb{C}^{n}$ as \( z = (z',z'') \) with
\[
z' = (z_{1},\dots, z_{n-m}) \in \mathbb{C}^{n-m}, \qquad
z'' = (z_{n-m+1},\dots, z_{n}) \in \mathbb{C}^{m}.
\]
Let \(\mathcal{R}_{m}\) denote  the restriction operator that maps  functions on   \( \mathbb{H}^{n} \) to \(\mathbb{H}^{n-m}\).
As shown in Lemma \ref{lm4.2},  \(\mathcal{R}_{m}\) extends to a bounded operator from \(W^{s,2}(\mathbb{H}^{n})\) to \(W^{s-2m,2}(\mathbb{H}^{n-m})\).

Correspondingly, we set
\begin{align}\label{d1.7}
\mathbb{S}^{2(n-m)+1} = \left\{\zeta \in \mathbb{S}^{2n+1} : \zeta_{k}=0, \; n-m+2 \leq k \leq n+1 \right\}.
\end{align}
Let \(\widetilde{\mathcal{R}}_{m}\)  be the operator that restricts  functions on  \( \mathbb{S}^{2n+1} \) to \(\mathbb{S}^{2(n-m)+1}\).
 For brevity, and when no confusion is likely, we also identify
 \begin{align}\label{d1.8}
\mathbb{S}^{2(n-m)+1} =& \Bigl\{(\zeta, t) : \sum_{j=1}^{n-m+1}|\zeta_{j}|^{2}=1,\;  \zeta_{k}=0, \; n-m+2 \leq k \leq n,\; t=0\Bigr\}\\
\subset& \mathbb{C}^{n-m+1}\times \{0\}\times\cdots \times\{0\}\subset  \mathbb{H}^{n}. \nonumber
 \end{align}
With this identification,  the notation \(\widetilde{\mathcal{R}}_{m}\) will also be used for the restriction of  functions on \(  \mathbb{H}^{n} \) to \(\mathbb{S}^{2(n-m)+1}\).

Employing the method of Bez--Machihara--Sugimoto together with the sharp form of the Hardy--Littlewood--Sobolev inequality on the Heisenberg group due to Frank and Lieb, we obtain the following three types of Sobolev trace inequalities on \(\mathbb{H}^{n}\) and $\mathbb{S}^{2n+1}$.

\begin{theorem}\label{th1.3}
Let \( 1 \leq m < n \), \( 2m < s < Q_n \), and \( p = \frac{2(Q_{n} - 2m)}{Q_{n} - s} \). For any \( f \in W^{s,2}(\mathbb{H}^{n}) \), it holds
\begin{align*}
\int_{\mathbb{H}^{n}}|\mathcal{L}_{s,\mathbb{H}^{n}}^{1/2}f|^{2} \geq &
2^{\frac{s-2m}{Q_n-2m}}\frac{\pi^{m}\Gamma\left(\frac{s}{2}\right)}{\Gamma\left(\frac{s-2m}{2}\right)}
\left( \frac{2\pi^{n-m+1}}{(n-m)!} \right)^{\frac{s-2m}{Q_n-2m}}
\frac{\Gamma^{2}\bigl(\frac{Q_n + s - 4m}{4}\bigr)}{\Gamma^{2}\bigl(\frac{Q_n-s}{4}\bigr)}
\|\mathcal{R}_m f\|_{L^{p}(\mathbb{H}^{n-m})}^{2}.
\end{align*}
Equality holds if and only if
\begin{align*}
f(u) = f(z',z'',t) = c|z''|^{s-2m}
\int_{\mathbb{H}^{n-m}} |u^{-1}v|^{4m - Q_n - s}
\Bigl\{J_{\mathscr{C}_{n-m}}\bigl(\delta_r(a^{-1}v)\bigr)\Bigr\}^{\frac{Q_{n} - s}{2(Q_{n} - 2m)}}  dv,
\end{align*}
for some \( c \in \mathbb{C} \), \( r > 0 \), and \( a \in \mathbb{H}^{n-m} \).
\end{theorem}

\begin{theorem}\label{th1.4}
Let \( 1 \leq m < n \), \( 2m < s < Q_n \), and \( p = \frac{2(Q_{n} - 2m)}{Q_{n} - s} \).
Let $\mathbb{S}^{2(n-m)+1}$ be as defined in  (\ref{d1.7}).
 For any \( F \in W^{s,2}(\mathbb{S}^{2n+1}) \), it holds
\begin{align*}
\int_{\mathbb{S}^{2n+1}}|\mathcal{A}_{s,\mathbb{S}^{2n+1}}^{1/2}F|^{2} \geq &
\frac{\pi^{m}\Gamma\left(\frac{s}{2}\right)}{\Gamma\left(\frac{s-2m}{2}\right)}
\left( \frac{2\pi^{n-m+1}}{(n-m)!} \right)^{\frac{s-2m}{Q_n-2m}}
\frac{\Gamma^{2}\bigl(\frac{Q_n + s - 4m}{4}\bigr)}{\Gamma^{2}\bigl(\frac{Q_n-s}{4}\bigr)}
\|\widetilde{\mathcal{R}}_m F\|_{L^{p}(\mathbb{S}^{2(n-m)+1})}^{2}.
\end{align*}
Equality holds if and only if
\begin{align*}
F(\zeta) = c\left(\sum_{j=n-m+2}^{n+1}|\zeta_{j}|^{2}\right)^{\frac{s-2m}{2}}\int_{\mathbb{S}^{2(n-m)+1}}
\frac{1}
{|1 - \zeta\cdot\overline{\eta}|^{\frac{Q_n + s - 4m}{2}}}
\frac{1}{|1 - \overline{\xi} \cdot \eta|^{\frac{Q_{n-m} - s}{2}}}  d\eta,
\end{align*}
where \( c \in \mathbb{R} \), \( \xi \in \mathbb{C}^{n-m+1} \) with \( |\xi| < 1 \), and \( \zeta \in \mathbb{S}^{2(n-m)+1} \).
\end{theorem}

\begin{theorem}\label{th1.5}
Let \( 1 \leq m < n \), \( 2m < s < Q_n \), and \( p = \frac{2(Q_{n} - 2m)}{Q_{n} - s} \). Let $\mathbb{S}^{2(n-m)+1}$ be  as defined in  (\ref{d1.8}). For any \( f \in W^{s,2}(\mathbb{H}^{n}) \), it holds
\begin{align*}
\int_{\mathbb{H}^{n}}|\mathcal{L}_{s,\mathbb{H}^{n}}^{1/2}f|^{2} \geq &2
\frac{\pi^{m}\Gamma\left(\frac{s}{2}\right)}{\Gamma\left(\frac{s-2m}{2}\right)}
\left( \frac{2\pi^{n-m+1}}{(n-m)!} \right)^{\frac{s-2m}{Q_n-2m}}
\frac{\Gamma^{2}\bigl(\frac{Q_n + s - 4m}{4}\bigr)}{\Gamma^{2}\bigl(\frac{Q_n-s}{4}\bigr)}
\|\widetilde{\mathcal{R}}_m f\|_{L^{p}(\mathbb{S}^{2(n-m)+1})}^{2}.
\end{align*}
Equality holds if and only if
\begin{align*}
f(u) = f(z',z'',t) = &  c|z''|^{s-2m}
\int_{\mathbb{S}^{2(n-m)+1}} |u^{-1}v|^{4m - Q_n - s}
\frac{1}{|1 - \overline{\xi} \cdot v|^{(Q_{n-m} - s)/2}}  dv,
\end{align*}
where $c \in \mathbb{R},\; \xi \in \mathbb{C}^{n-m+1},\; |\xi| < 1$.
\end{theorem}

We now consider the limiting case \( s = Q_n \). At this endpoint, it becomes necessary to introduce the operator \( \mathcal{A}'_{Q_n,\mathbb{S}^{2n+1}} \), which was originally defined by Branson, Fontana and Morpurgo \cite{br1} on the space of CR-pluriharmonic functions.

Let us denote
\[
\mathcal{P}_{n} = \bigoplus_{j>0} \bigl(\mathcal{H}_{j,0} \oplus \mathcal{H}_{0,j}\bigr) \oplus \mathcal{H}_{0,0}
= \bigl\{ L^2 \text{ CR-pluriharmonic functions} \bigr\};
\]
and
\[
\mathbb{R}\mathcal{P}_{n} = \bigl\{ L^2 \text{ real-valued CR-pluriharmonic functions} \bigr\}.
\]
Define
\begin{align}\nonumber
\mathcal{A}'_{Q_n,\mathbb{S}^{2n+1}} F
=& -\frac{4}{\Gamma\bigl(\frac{Q_n}{2}\bigr)}
\frac{\partial}{\partial s}\Big|_{s=Q_n} \mathcal{A}_{s,\mathbb{S}^{2n+1}} F \\
\label{1.7}
=& -\frac{4}{\Gamma\bigl(\frac{Q_n}{2}\bigr)}
\lim_{s\to Q_n} \frac{1}{s-Q_n}\mathcal{A}_{s,\mathbb{S}^{2n+1}} F,
\qquad F \in C^{\infty}(\mathbb{S}^{2n+1}) \cap \mathcal{P}_{n}.
\end{align}
Branson, Fontana and Morpurgo proved the following Beckner--Onofri-type inequality for \( \mathcal{A}'_{Q_n,\mathbb{S}^{2n+1}} \):
\begin{theorem}\label{th3.1} \cite{br1} \label{brl}
For any \( F \in W^{Q_n/2,2} \cap \mathbb{R}\mathcal{P}_{n} \), it holds
\begin{align}\label{1.8}
\frac{1}{2(n+1)!|\mathbb{S}^{2n+1}|}
\int_{\mathbb{S}^{2n+1}} F\mathcal{A}_{Q_n}'F
+\frac{1}{|\mathbb{S}^{2n+1}|}\int_{\mathbb{S}^{2n+1}} F \geq \log\Bigl(\frac{1}{|\mathbb{S}^{2n+1}|}\int_{\mathbb{S}^{2n+1}} e^{F}\Bigr).
\end{align}
Equality holds only for functions of the form
\[
- Q_n \ln|1 - \overline{\xi} \cdot \zeta| + c, \qquad
c \in \mathbb{R},\; \xi \in \mathbb{C}^{n+1},\; |\xi| < 1,\; \zeta \in \mathbb{S}^{2n+1}.
\]
\end{theorem}

Our trace  Beckner--Onofri inequality at the endpoint \( s = Q_n \) reads as follows:

\begin{theorem}\label{th1.7}
Let \( 1 \leq m < n \). For any \( F \in W^{Q_n/2,2} \cap \mathbb{R}\mathcal{P}_{n} \), it holds
\begin{align*}
&\frac{1}{2(n-m+1)!|\mathbb{S}^{2(n-m)+1}|}
\int_{\mathbb{S}^{2n+1}} F\mathcal{A}_{Q_n}'F+\frac{\pi^{m}}{|\mathbb{S}^{2(n-m)+1}|}\int_{\mathbb{S}^{2(n-m)+1}} \widetilde{\mathcal{R}}_mF  \\
\geq&
\pi^{m} \log\Bigl(\frac{1}{|\mathbb{S}^{2n+1}|}\int_{\mathbb{S}^{2n+1}} e^{\widetilde{\mathcal{R}}_mF}\Bigr)
\end{align*}
Equality holds if and only if
\begin{align*}
F(\zeta)
=
\frac{\Gamma\bigl(\frac{Q_n-2m}{2}\bigr)}
{2\pi^{n-m+1}}
\int_{\mathbb{S}^{2(n-m)+1}} \frac{\bigl(\sum\limits_{j=n-m+2}^{n+1}|\zeta_{j}|^{2}\bigr)^{\frac{Q_n-2m}{2}}}
{|1-\zeta\cdot\overline{\eta}|^{Q_n-2m}}\Bigl(-Q_{n-m}\ln|1 - \overline{\xi} \cdot \eta| + c\Bigr)  d\eta,
\end{align*}
where \( c \in \mathbb{R} \), \( \xi \in \mathbb{C}^{n-m+1} \) with \( |\xi| < 1 \), and \( \zeta \in \mathbb{S}^{2(n-m)+1} \).
\end{theorem}

Our method further allows us to derive trace versions of the Beckner--Onofri inequalities on the standard sphere, a setting that has not been previously considered in the literature \cite{ei,be2,bez}.

Let \(\mathbb{S}^{n}\) denote the standard sphere
\[
\mathbb{S}^{n} = \{ x \in \mathbb{R}^{n+1} : |x| = 1 \}.
\]
Denote by \(H^{s}(\mathbb{S}^{n})\) the completion of \(C^\infty(\mathbb{S}^{n})\) with respect to the norm
\[
\| F \|_{H^{s}(\mathbb{S}^{n})} = \| (-\Delta_{\mathbb{S}^{n}})^{s/2} F \|_{L^{2}(\mathbb{S}^{n})},
\]
where \(\Delta_{\mathbb{S}^{n}}\) is the Laplace--Beltrami operator on \(\mathbb{S}^{n}\).
The classical Beckner--Onofri inequality on \(\mathbb{S}^{n}\) reads (see \cite{on,be1}):

\begin{theorem}  \label{Beckner}
Define
\begin{equation*}
P_{s,\mathbb{S}^{n}} = \frac{\Gamma\bigl(B + \frac{1+s}{2} \bigr)}{\Gamma\bigl(B + \frac{1-s}{2} \bigr)}, \qquad
B = \sqrt{-\Delta_{\mathbb{S}^{n}} + \frac{(n-1)^{2}}{4}} .
\end{equation*}
Then, for \( s = n \) and any \( F \in H^{n}(\mathbb{S}^{n}) \),
\begin{equation}\label{1.9}
\ln\Bigl( \frac{1}{|\mathbb{S}^{n}|} \int_{\mathbb{S}^{n}} e^{F} \Bigr) \leq
\frac{1}{2n!|\mathbb{S}^{n}|} \int_{\mathbb{S}^{n}} F  P_{n,\mathbb{S}^{n}} F
+ \frac{1}{|\mathbb{S}^{n}|} \int_{\mathbb{S}^{n}} F .
\end{equation}
Equality holds only for functions of the form
\[
- n \ln| 1 -  \xi \cdot \zeta | + c, \qquad
\xi \in \mathbb{R}^{n+1},\;|\xi|<1,\;  \zeta \in \mathbb{S}^{n},\; c \in \mathbb{R}.
\]
\end{theorem}

Define the lower-dimensional sphere
\[
\mathbb{S}^{n-m} = \bigl\{ (\xi_{1},\dots, \xi_{n+1}) \in \mathbb{S}^{n} : \xi_{n-m+2} = \cdots = \xi_{n+1} = 0 \bigr\},
\]
and denote by \(\widetilde{\tau}_{m}\) the restriction operator mapping functions \(F \in H^{s}(\mathbb{S}^{n})\) to \(H^{s}(\mathbb{S}^{n-m})\).

Our trace  Beckner--Onofri inequality on \(\mathbb{S}^{n}\) is as follows.

\begin{theorem}\label{th1.8}
Let \( 1 \leq m < n \). For any \( F \in H^{n}(\mathbb{S}^{n}) \), we have
\begin{align}\nonumber
&\frac{1}{2(n-m)!|\mathbb{S}^{n-m}|} \int_{\mathbb{S}^{n}} F  P_{n,\mathbb{S}^{n}} F
+ \pi^{m/2}  2^{m}  \frac{\Gamma\left(\frac{n}{2}\right)}{\Gamma\left(\frac{n-m}{2}\right)|\mathbb{S}^{n-m}|} \int_{\mathbb{S}^{n-m}} \widetilde{\tau}_{m}F \\
\label{eq:trace-ineq2}
\geq&\pi^{m/2}  2^{m}  \frac{\Gamma\left(\frac{n}{2}\right)}{\Gamma\left(\frac{n-m}{2}\right)}\ln\Bigl( \frac{1}{|\mathbb{S}^{n-m}|} \int_{\mathbb{S}^{n-m}} e^{\widetilde{\tau}_{m}F} \Bigr).
\end{align}
Equality holds only if
\begin{align*}
F(\zeta)=&\pi^{\frac{m-n}{2}}\frac{\Gamma\bigl(n-m\bigr)}{\Gamma(\frac{n-m}{2})}\Bigl(\sum_{j=n-m+2}^{n+1}|\zeta_{j}|^{2}\Bigr)^{\frac{n-m}{2}}
\int_{\mathbb{S}^{n-m}} \frac{- (n-m) \ln| 1 -  \xi\cdot\eta | + c}{|\zeta-\eta|^{2n-2m}}
d\eta,
\end{align*}
where $\xi \in \mathbb{R}^{n+1},\;|\xi|<1,\;  \xi \in \mathbb{S}^{n},\; c \in \mathbb{R}.$
\end{theorem}

This paper is organized as follows. In Section 2, we collect the necessary preliminary material, including the Cayley transform and the sharp Sobolev inequalities on both the Heisenberg group and the CR sphere. Section 3 is devoted to the proof of our main results: Theorems \ref{th1.3}, \ref{th1.4}, \ref{th1.5} and \ref{th1.7}. Finally, in Section 4 we establish the trace versions of the Beckner-Onofri inequalities on the standard sphere, i.e. Theorem \ref{th1.8}.

\section{Preliminaries}
We begin by recalling some preliminary facts that will be used in the subsequent analysis.

\subsection{Cayley transform}
The Cayley transform \(\mathscr{C}_{n} : \mathbb{H}^n \to \mathbb{S}^{2n+1}\) and its inverse \(\mathscr{C}_{n}^{-1} : \mathbb{S}^{2n+1} \to \mathbb{H}^n\) are defined by
\begin{align}
\label{2.1}
\mathscr{C}_{n}(z, t) =& \Bigl( \frac{2z}{1 + |z|^2 + it},\; \frac{1 - |z|^2 - it}{1 + |z|^2 + it} \Bigr),\\
\label{2.2}
\mathscr{C}_{n}^{-1}(\zeta) =& \Bigl( \frac{\zeta_1}{1 + \zeta_{n+1}},\dots, \frac{\zeta_n}{1 + \zeta_{n+1}},\;
\operatorname{Im}\frac{1 - \zeta_{n+1}}{1 + \zeta_{n+1}} \Bigr).
\end{align}
The Jacobian of this transformation (see \cite{br1}) is
\[
|J_{\mathscr{C}_{n}}(z, t)| = \frac{2^{2n+1}}{\bigl( (1 + |z|^2)^2 + t^2 \bigr)^{n+1}}.
\]
Observe that
\begin{align}\label{2.3}
2|J_{\mathscr{C}_{n}}(z, t)| = \left(\frac{4}{(1 + |z|^2)^2 + t^2 }\right)^{\frac{Q_n}{2}}.
\end{align}
Via the Cayley transform, the operators \(\mathcal{L}_{s,\mathbb{H}^{n}}\) and \(\mathcal{A}_{s,\mathbb{S}^{2n+1}}\) are related as follows (see \cite{br1}): for \(0 < s < Q_n\) and \(F \in C^{\infty}(\mathbb{S}^{2n+1})\),
\begin{align}\label{2.6}
\mathcal{L}_{s,\mathbb{H}^{n}}\Bigl( (2|J_{\mathscr{C}_{n}}|)^{\frac{Q_n-s}{2Q_n}} (F \circ \mathscr{C}_{n}) \Bigr)
= (2|J_{\mathscr{C}_{n}}|)^{\frac{Q_n+s}{2Q_n}} (\mathcal{A}_{s,\mathbb{S}^{2n+1}} F) \circ \mathscr{C}_{n}.
\end{align}
In particular, choosing \(F \equiv 1\) yields
\begin{align}\label{2.7}
\mathcal{L}_{s,\mathbb{H}^{n}} (2|J_{\mathscr{C}_{n}}|)^{\frac{Q_n-s}{2Q_n}}
= \frac{\Gamma^{2}\bigl(\frac{Q_n + s}{4}\bigr)}{\Gamma^{2}\bigl(\frac{Q_n-s}{4}\bigr)}
(2|J_{\mathscr{C}_{n}}|)^{\frac{Q_n+s}{2Q_n}}, \qquad 0 < s < Q_n.
\end{align}

\subsection{Sobolev inequalities on \(\mathbb{H}^{n}\) and \(\mathbb{S}^{2n+1}\)}

We first verify formula \eqref{As}. For any spherical harmonic \(Y_{j,k} \in \mathcal{H}_{j,k}\), we have (see \cite{st,br1})
\begin{align*}
\mathcal{D} Y_{j,k} = \bigl(j + \tfrac{n}{2}\bigr)\bigl(k + \tfrac{n}{2}\bigr), \qquad
\mathcal{T} Y_{j,k} = \tfrac{i}{2}(j - k) Y_{j,k}.
\end{align*}
Consequently,
\[
\sqrt{\mathcal{D} - \mathcal{T}^{2}}  Y_{j,k}
= \sqrt{\bigl(j + \tfrac{n}{2}\bigr)\bigl(k + \tfrac{n}{2}\bigr) + \tfrac{(j-k)^{2}}{4}}  Y_{j,k}
= \tfrac{j + k + n}{2}  Y_{j,k}.
\]
Hence,
\begin{align*}
& \frac{\Gamma\Bigl(\sqrt{\mathcal{D} - \mathcal{T}^{2}} - i\mathcal{T} + \tfrac{2+s}{4}\Bigr)
      \Gamma\Bigl(\sqrt{\mathcal{D} - \mathcal{T}^{2}} + i\mathcal{T} + \tfrac{2+s}{4}\Bigr)}
     {\Gamma\Bigl(\sqrt{\mathcal{D} - \mathcal{T}^{2}} - i\mathcal{T} + \tfrac{2-s}{4}\Bigr)
      \Gamma\Bigl(\sqrt{\mathcal{D} - \mathcal{T}^{2}} + i\mathcal{T} + \tfrac{2-s}{4}\Bigr)}  Y_{j,k} \\
=& \frac{\Gamma\Bigl(\tfrac{j+k+n}{2} + \tfrac{j-k}{2} + \tfrac{2+s}{4}\Bigr)
      \Gamma\Bigl(\tfrac{j+k+n}{2} - \tfrac{j-k}{2} + \tfrac{2+s}{4}\Bigr)}
     {\Gamma\Bigl(\tfrac{j+k+n}{2} + \tfrac{j-k}{2} + \tfrac{2-s}{4}\Bigr)
      \Gamma\Bigl(\tfrac{j+k+n}{2} - \tfrac{j-k}{2} + \tfrac{2-s}{4}\Bigr)}  Y_{j,k} \\
=& \frac{\Gamma\bigl(\tfrac{Q_n+s}{4} + j\bigr)}{\Gamma\bigl(\tfrac{Q_n-s}{4} + j\bigr)}\;
   \frac{\Gamma\bigl(\tfrac{Q_n+s}{4} + k\bigr)}{\Gamma\bigl(\tfrac{Q_n-s}{4} + k\bigr)}
   Y_{j,k} \\
=& \mathcal{A}_{s,\mathbb{S}^{2n+1}} Y_{j,k},
\end{align*}
which establishes \eqref{As}.

Observe also the relation
\[
-\tfrac{1}{4}\Delta_{\mathbb{S}^{2n+1}} = -\mathcal{L} - \mathcal{T}^{2},
\]
where \(\Delta_{\mathbb{S}^{2n+1}}\) is the Laplace--Beltrami operator on \(\mathbb{S}^{2n+1}\). Indeed, for any \(Y_{j,k} \in \mathcal{H}_{j,k}\),
\begin{align*}
\bigl(-\mathcal{L} - \mathcal{T}^{2}\bigr) Y_{j,k}
=& \Bigl[\bigl(j + \tfrac{n}{2}\bigr)\bigl(k + \tfrac{n}{2}\bigr) - \tfrac{n^{2}}{4} + \tfrac{(j-k)^{2}}{4}\Bigr] Y_{j,k} \\
=& \tfrac{1}{4}(j+k)(j+k+2n) = -\tfrac{1}{4}\Delta_{\mathbb{S}^{2n+1}} Y_{j,k}.
\end{align*}
Thus, an alternative representation of \(\mathcal{A}_{s,\mathbb{S}^{2n+1}}\) is
\[
\mathcal{A}_{s,\mathbb{S}^{2n+1}} =
\frac{\Gamma\Bigl(\tfrac{1}{2}\sqrt{-\Delta_{\mathbb{S}^{2n+1}}+n^{2}} - i\mathcal{T} + \tfrac{2+s}{4}\Bigr)
      \Gamma\Bigl(\tfrac{1}{2}\sqrt{-\Delta_{\mathbb{S}^{2n+1}}+n^{2}} + i\mathcal{T} + \tfrac{2+s}{4}\Bigr)}
     {\Gamma\Bigl(\tfrac{1}{2}\sqrt{-\Delta_{\mathbb{S}^{2n+1}}+n^{2}} - i\mathcal{T} + \tfrac{2-s}{4}\Bigr)
      \Gamma\Bigl(\tfrac{1}{2}\sqrt{-\Delta_{\mathbb{S}^{2n+1}}+n^{2}} + i\mathcal{T} + \tfrac{2-s}{4}\Bigr)}.
\]

In a ground-breaking work \cite{fl2}, Frank and Lieb proved the following sharp Hardy--Littlewood--Sobolev inequality on \(\mathbb{S}^{2n+1}\):
\begin{theorem}\cite{fl2} \label{HLS2}
Let $0 < \lambda < Q_n = 2n + 2$ and $p= 2Q_n/(2Q_n - \lambda)$. Then for any $f,g \in L^p(\mathbb{S}^{2n+1})$
\[
\left| \iint_{\mathbb{S}^{2n+1} \times \mathbb{S}^{2n+1}} \frac{\overline{f(\zeta)} g(\eta)}{|1 - \zeta \cdot \overline{\eta}|^{\lambda/2}} d\zeta d\eta \right|
\leq \left( \frac{2\pi^{n+1}}{n!} \right)^{\lambda/Q_n} \frac{n! \Gamma\bigl((Q_n - \lambda)/2\bigr)}{\Gamma^2\bigl((2Q_n - \lambda)/4\bigr)} \|f\|_p \|g\|_p
\]
with equality if and only if
\[
f(\zeta) = \frac{c}{|1 - \overline{\xi} \cdot \zeta|^{(2Q_n-\lambda)/2}}, \quad g(\zeta) = \frac{c'}{|1 - \overline{\xi} \cdot \zeta|^{(2Q_n-\lambda)/2}},
\]
for some $c, c' \in \mathbb{C}$ and some $\xi \in \mathbb{C}^{n+1}$ with $|\xi| < 1$.
\end{theorem}
Following an argument of Beckner (see \cite[Theorem 6]{be1}), one can  deduce the  following sharp Sobolev inequality on \(\mathbb{S}^{2n+1}\) (see \cite{fl2}, (5.13)):
\begin{align}\label{Sobolev inequalities1}
\int_{\mathbb{S}^{2n+1}} |\mathcal{A}_{s,\mathbb{S}^{2n+1}}^{1/2}F|^{2} \geq& \left( \frac{2\pi^{n+1}}{n!} \right)^{s/Q_n} \frac{\Gamma^{2}(\frac{Q_n +s}{4})}{\Gamma^2(\frac{Q_n-s}{4})} \|F\|_{L^{\frac{2Q_n}{Q_n-s}}(\mathbb{S}^{2n+1})}^{2},\\
\nonumber
&\quad  0<s<Q_n,\; F\in W^{s,2}(\mathbb{S}^{2n+1}),
\end{align}
with equality if and only if
\[
F(\zeta) = \frac{c}{|1 - \overline{\xi} \cdot \zeta|^{(Q_n - s)/2}}, \qquad
c \in \mathbb{R},\; \xi \in \mathbb{C}^{n+1},\; |\xi| < 1,\; \zeta \in \mathbb{S}^{2n+1}.
\]

On the other hand, from \eqref{2.6} we obtain the identity
\begin{align}\label{a2.6}
\int_{\mathbb{H}^{n}}\Bigl|\mathcal{L}_{s,\mathbb{H}^{n}}^{1/2}
\bigl( (2|J_{\mathscr{C}_{n}}|)^{\frac{Q_n-s}{2Q_n}} (F \circ \mathscr{C}_{n}) \bigr)\Bigr|^{2}
= 2\int_{\mathbb{S}^{2n+1}} |\mathcal{A}_{s,\mathbb{S}^{2n+1}}^{1/2} F|^{2},
\;\; F \in W^{s,2}(\mathbb{S}^{2n+1}).
\end{align}
Consequently, the sharp Sobolev inequality on \(\mathbb{H}^{n}\) reads (see \cite{Cha} for $0<s\leq 1$)
\begin{align}\label{Sobolev inequalities2}
\int_{\mathbb{H}^{n}}|\mathcal{L}_{s,\mathbb{H}^{n}}^{1/2}f|^{2}
\geq& 2^{s/Q_n} \left( \frac{2\pi^{n+1}}{n!} \right)^{s/Q_n}
\frac{\Gamma^{2}\bigl(\frac{Q_n + s}{4}\bigr)}{\Gamma^{2}\bigl(\frac{Q_n - s}{4}\bigr)}
\|f\|_{L^{\frac{2Q_n}{Q_n - s}}(\mathbb{H}^{n})}^{2},\\
\nonumber
&\qquad 0 < s < Q_n,\; f \in W^{s,2}(\mathbb{H}^{n}),
\end{align}
with equality if and only if
\[
f(u) = c \Bigl\{ J_{\mathscr{C}_{n}}\bigl(\delta_{r}(v^{-1}u)\bigr) \Bigr\}^{\frac{Q_n - s}{2Q_n}}, \qquad
c \in \mathbb{C},\; r > 0,\; v \in \mathbb{H}^n.
\]
For other fractional subelliptic Sobolev inequalities on $\mathbb{H}^{n}$, we refer to \cite{Gho} and the references therein.

\section{Proof of Theorem \ref{th1.3}, \ref{th1.4}, \ref{th1.5} and \ref{th1.7}}

Define
\[
k_{s} := \mathcal{L}_{s,\mathbb{H}^{n}}^{-1/2}\delta_0, \qquad 0<s<Q_n,
\]
where \(\delta_0\) denotes the Dirac delta at the origin.
Frank and Lieb (see \cite{fl2}, Proposition 4.1) proved that \(k_s\) is real-valued, even (i.e., \(k_s(u^{-1})=k_s(u)\) for all \(u\in\mathbb{H}^n\)), and homogeneous of degree \((s-2Q_n)/2\) in the sense that
\[
k_s(\delta_r u) = r^{\frac{s-2Q_n}{2}} k_s(u), \qquad u\in\mathbb{H}^n,\; r>0.
\]
Moreover, \(k_s\) belongs to the weak Lebesgue space \(L^{2Q_n/(2Q_n-s)}_w(\mathbb{H}^n)\).
Note that (see \cite{Co,ron1})
\begin{align}\label{3.1}
(\mathcal{L}_{s,\mathbb{H}^{n}}^{-1}\delta_0)(u) = a_{n,s} |u|^{s-Q_n}, \qquad
a_{n,s} = \frac{2^{n-1-\frac{s}{2}}\Gamma^2\bigl(\frac{Q_n-s}{4}\bigr)}{\pi^{n+1}\Gamma(\frac{s}{2})},
\end{align}
from which one deduces the convolution identity
\begin{align}\label{3.2}
\int_{\mathbb{H}^n} k_s(u^{-1}w)k_s(v^{-1}w)dw = a_{n,s}|u^{-1}v|^{s-Q_n}, \qquad u,v\in\mathbb{H}^n.
\end{align}

Following \cite{bez}, we introduce the linear operator
\[
\mathcal{S}_{m}: C_{0}^{\infty}(\mathbb{H}^{n})\subset L^2(\mathbb{H}^{n}) \longrightarrow L^2(\mathbb{H}^{n-m})
\]
defined by
\[
\mathcal{S}_{m}f = \mathcal{L}_{s-2m,\mathbb{H}^{n-m}}^{1/2}\mathcal{R}_{m}
\mathcal{L}_{s,\mathbb{H}^{n}}^{-1/2}f
= \mathcal{L}_{s-2m,\mathbb{H}^{n-m}}^{1/2}\mathcal{R}_{m}(f\ast k_s),
\qquad f\in C_{0}^{\infty}(\mathbb{H}^{n}).
\]
For simplicity, we restrict our attention to real-valued functions in what follows.
Let \(\mathcal{S}_{m}^{*}: L^2(\mathbb{H}^{n-m})\to L^2(\mathbb{H}^{n})\) be the adjoint operator of \(\mathcal{S}_{m}\). Then
\[
\mathcal{S}_{m}^{*}g(u)=\int_{\mathbb{H}^{n-m}}k_s(v^{-1}u)
\mathcal{L}_{s-2m,\mathbb{H}^{n-m}}^{1/2}g(v)dv,\qquad g\in L^2(\mathbb{H}^{n-m}).
\]
Consequently,
\begin{align}\nonumber
\mathcal{S}_{m}\mathcal{S}_{m}^{*}g(u)
=& \mathcal{L}_{s-2m,\mathbb{H}^{n-m}}^{1/2}\mathcal{R}_{m}
\int_{\mathbb{H}^{n-m}}(\mathcal{L}_{s,\mathbb{H}^{n}}^{-1/2}k_s)(v^{-1}u)
\mathcal{L}_{s-2m,\mathbb{H}^{n-m}}^{1/2}g(v)dv \\
\label{3.3}
=& a_{n,s}\mathcal{L}_{s-2m,\mathbb{H}^{n-m}}^{1/2}\mathcal{R}_{m}
\int_{\mathbb{H}^{n-m}}|v^{-1}u|^{s-Q_n}
\mathcal{L}_{s-2m,\mathbb{H}^{n-m}}^{1/2}g(v)dv.
\end{align}
Inserting (\ref{3.1}) into (\ref{3.3}) gives
\begin{align}\nonumber
\mathcal{S}_{m}\mathcal{S}_{m}^{*}g
=& a_{n,s}\mathcal{L}_{s-2m,\mathbb{H}^{n-m}}^{1/2}
\Bigl(\frac{1}{a_{n-m,s-2m}}
\mathcal{L}_{s-2m,\mathbb{H}^{n-m}}^{-1}
\mathcal{L}_{s-2m,\mathbb{H}^{n-m}}^{1/2}g\Bigr)\\
\nonumber
=&\frac{a_{n,s}}{a_{n-m,s-2m}}\; g= \frac{2^{n-1-\frac{s}{2}}\Gamma^2\bigl(\frac{Q_n-s}{4}\bigr)}
{\pi^{n+1}\Gamma(\frac{s}{2})}\;
\frac{\pi^{n-m+1}\Gamma(\frac{s-2m}{2})}
{2^{n-m-1-\frac{s-2m}{2}}\Gamma^2\bigl(\frac{Q_{n-m}-(s-2m)}{4}\bigr)}\; g\\
\label{3.4}
=& \frac{\Gamma\bigl(\frac{s-2m}{2}\bigr)}{\pi^{m}\Gamma\bigl(\frac{s}{2}\bigr)}\; g .
\end{align}

\begin{lemma}\label{lm4.1}
For any \(f\in C_{0}^{\infty}(\mathbb{H}^{n})\) and \(0<s< Q_n\),
\begin{align*}
\int_{\mathbb{H}^{n}}\bigl|\mathcal{L}_{s,\mathbb{H}^{n}}^{1/2}f\bigr|^{2}
=& \int_{\mathbb{H}^{n}}\Bigl|\mathcal{L}_{s,\mathbb{H}^{n}}^{1/2}\Bigl(f-
\frac{\pi^{m}\Gamma\bigl(\frac{s}{2}\bigr)}{\Gamma\bigl(\frac{s-2m}{2}\bigr)}
\mathcal{L}_{s,\mathbb{H}^{n}}^{-1/2}\mathcal{S}_{m}^{\ast}
\mathcal{L}_{s-2m,\mathbb{H}^{n-m}}^{1/2}\mathcal{R}_{m}f\Bigr)\Bigr|^{2} +\\
&\qquad \frac{\pi^{m}\Gamma\bigl(\frac{s}{2}\bigr)}{\Gamma\bigl(\frac{s-2m}{2}\bigr)}
\int_{\mathbb{H}^{n-m}}\bigl|\mathcal{L}_{s-2m,\mathbb{H}^{n-m}}^{1/2}\mathcal{R}_{m}f\bigr|^{2}.
\end{align*}
\end{lemma}

\begin{proof}
Set \(h=\mathcal{L}_{s,\mathbb{H}^{n}}^{1/2}f\); then \(f=\mathcal{L}_{s,\mathbb{H}^{n}}^{-1/2}h=h\ast k_s\). We compute
\begin{align*}
&\int_{\mathbb{H}^{n}}\Bigl|\mathcal{L}_{s,\mathbb{H}^{n}}^{1/2}\Bigl(f-
\frac{\pi^{m}\Gamma\bigl(\frac{s}{2}\bigr)}{\Gamma\bigl(\frac{s-2m}{2}\bigr)}
\mathcal{L}_{s,\mathbb{H}^{n}}^{-1/2}\mathcal{S}_{m}^{\ast}
\mathcal{L}_{s-2m,\mathbb{H}^{n-m}}^{1/2}\mathcal{R}_{m}f\Bigr)\Bigr|^{2} \\
=& \int_{\mathbb{H}^{n}}\Bigl|h-
\frac{\pi^{m}\Gamma\bigl(\frac{s}{2}\bigr)}{\Gamma\bigl(\frac{s-2m}{2}\bigr)}
\mathcal{S}_{m}^{\ast}\mathcal{L}_{s-2m,\mathbb{H}^{n-m}}^{1/2}
\mathcal{R}_{m}\mathcal{L}_{s,\mathbb{H}^{n}}^{-1/2}h\Bigr|^{2} \\
=& \int_{\mathbb{H}^{n}}\Bigl|h-
\frac{\pi^{m}\Gamma\bigl(\frac{s}{2}\bigr)}{\Gamma\bigl(\frac{s-2m}{2}\bigr)}
\mathcal{S}_{m}^{\ast}\mathcal{S}_{m}h\Bigr|^{2} \\
=& \int_{\mathbb{H}^{n}}|h|^{2}
+ \Bigl(\frac{\pi^{m}\Gamma\bigl(\frac{s}{2}\bigr)}{\Gamma\bigl(\frac{s-2m}{2}\bigr)}\Bigr)^{2}
\int_{\mathbb{H}^{n}}|\mathcal{S}_{m}^{\ast}\mathcal{S}_{m}h|^{2}
- 2\frac{\pi^{m}\Gamma\bigl(\frac{s}{2}\bigr)}{\Gamma\bigl(\frac{s-2m}{2}\bigr)}
\int_{\mathbb{H}^{n}}h\mathcal{S}_{m}^{\ast}\mathcal{S}_{m}h \\
=& \int_{\mathbb{H}^{n}}|h|^{2}
+ \Bigl(\frac{\pi^{m}\Gamma\bigl(\frac{s}{2}\bigr)}{\Gamma\bigl(\frac{s-2m}{2}\bigr)}\Bigr)^{2}
\int_{\mathbb{H}^{n}}|\mathcal{S}_{m}^{\ast}\mathcal{S}_{m}h|^{2}
- 2\frac{\pi^{m}\Gamma\bigl(\frac{s}{2}\bigr)}{\Gamma\bigl(\frac{s-2m}{2}\bigr)}
\int_{\mathbb{H}^{n-m}}|\mathcal{S}_{m}h|^{2}.
\end{align*}
On the other hand, using \eqref{3.4},
\[
\int_{\mathbb{H}^{n}}|\mathcal{S}_{m}^{\ast}\mathcal{S}_{m}h|^{2}
= \int_{\mathbb{H}^{n-m}}\mathcal{S}_{m}\mathcal{S}_{m}^{\ast}\mathcal{S}_{m}h\cdot\mathcal{S}_{m}h
= \frac{\Gamma\bigl(\frac{s-2m}{2}\bigr)}{\pi^{m}\Gamma\bigl(\frac{s}{2}\bigr)}
\int_{\mathbb{H}^{n-m}}|\mathcal{S}_{m}h|^{2}.
\]
Hence,
\begin{align*}
&\int_{\mathbb{H}^{n}}\Bigl|\mathcal{L}_{s,\mathbb{H}^{n}}^{1/2}f-
\frac{\pi^{m}\Gamma\bigl(\frac{s}{2}\bigr)}{\Gamma\bigl(\frac{s-2m}{2}\bigr)}
\mathcal{S}_{m}^{\ast}\mathcal{L}_{s-2m,\mathbb{H}^{n-m}}^{1/2}\mathcal{R}_{m}f\Bigr|^{2} \\
=& \int_{\mathbb{H}^{n}}|h|^{2}
- \frac{\pi^{m}\Gamma\bigl(\frac{s}{2}\bigr)}{\Gamma\bigl(\frac{s-2m}{2}\bigr)}
\int_{\mathbb{H}^{n-m}}|\mathcal{S}_{m}h|^{2},
\end{align*}
which is precisely the statement of the lemma.
\end{proof}

Now  we derive an explicit expression for
\[
\frac{\pi^{m}\Gamma\bigl(\frac{s}{2}\bigr)}{\Gamma\bigl(\frac{s-2m}{2}\bigr)}
\mathcal{L}_{s,\mathbb{H}^{n}}^{-1/2}\mathcal{S}_{m}^{\ast}\mathcal{L}_{s-2m,\mathbb{H}^{n-m}}^{1/2}\mathcal{R}_{m}f.
\]
First,
\begin{align}\nonumber
&\mathcal{L}_{s,\mathbb{H}^{n}}^{-1/2}\mathcal{S}_{m}^{\ast}\mathcal{L}_{s-2m,\mathbb{H}^{n-m}}^{1/2}\mathcal{R}_{m}f(u) \\
\nonumber
=& \mathcal{L}_{s,\mathbb{H}^{n}}^{-1/2}\int_{\mathbb{H}^{n-m}}k_s(v^{-1}u)
\mathcal{L}_{s-2m,\mathbb{H}^{n-m}}\mathcal{R}_{m}f(v)dv \\
\nonumber
=& \int_{\mathbb{H}^{n-m}}(\mathcal{L}_{s,\mathbb{H}^{n}}^{-1/2}k_s)(v^{-1}u)
\mathcal{L}_{s-2m,\mathbb{H}^{n-m}}\mathcal{R}_{m}f(v)dv \\
\label{4.2}
=& a_{n,s}\int_{\mathbb{H}^{n-m}}|v^{-1}u|^{s-Q_n}
\mathcal{L}_{s-2m,\mathbb{H}^{n-m}}\mathcal{R}_{m}f(v)dv .
\end{align}
Next,  we compute \(\mathcal{L}_{s-2m,\mathbb{H}^{n-m}}|v^{-1}u|^{s-Q_n}\).
For simplicity, write \(u=(z',z'',t)\) and \(v=(w',w'',\mu)\) with \(w''=0\). Then
\begin{align*}
|u^{-1}v|^{s-Q_n} =& \bigl[(|z''|^{2}+|-z'+w'|^{2})^{2}+(-t+\mu-2\operatorname{Im} z' \cdot \overline{w}')^{2}\bigr]^{\frac{s-Q_n}{4}} .
\end{align*}
If we set \(\widetilde{u}=(z',t)\) and \(\widetilde{v}=(w',\mu)\), then
\begin{align*}
|u^{-1}v|^{s-Q_n}
=& |z''|^{s-Q_n}
\Bigl[\Bigl(1+\frac{|-z'+w'|^{2}}{|z''|^{2}}\Bigr)^{2}
+ \frac{1}{|z''|^{4}}\bigl(-t+\mu-2\operatorname{Im} z' \cdot \overline{w}'\bigr)^{2}\Bigr]^{\frac{s-Q_n}{4}} \\
=& |z''|^{s-Q_n}
\Bigl[2^{-1-2(n-m)}J_{\mathscr{C}_{n-m}}\bigl(\delta_{\frac{1}{|z''|}}(\widetilde{u}^{-1}\widetilde{v})\bigr)\Bigr]
^{\frac{Q_{n-m}-(s-2m)}{2Q_{n-m}}} .
\end{align*}
Using (\ref{1.4}) and (\ref{2.7}), we obtain
\begin{align}\nonumber
&\mathcal{L}_{s-2m,\mathbb{H}^{n-m}}|u^{-1}v|^{s-Q_n} \\
\nonumber
=& \frac{\Gamma^{2}\bigl(\frac{Q_{n-m}+s-2m}{4}\bigr)}
{\Gamma^{2}\bigl(\frac{Q_{n-m}-(s-2m)}{4}\bigr)}
|z''|^{s-Q_n}\frac{1}{|z''|^{s-2m}}
\Bigl[2^{-2-2(n-m)}\Bigr]^{\frac{Q_{n-m}-(s-2m)}{2Q_{n-m}}}\\
\nonumber
&\qquad \times \bigl[2J_{\mathscr{C}_{n-m}}(\delta_{\frac{1}{|z''|}}(\widetilde{u}^{-1}\widetilde{v}))\bigr]^{\frac{Q_{n-m}+(s-2m)}{2Q_{n-m}}}\\
\nonumber
=& \frac{\Gamma^{2}\bigl(\frac{Q_{n}+s-4m}{4}\bigr)}
{\Gamma^{2}\bigl(\frac{Q_{n}-s}{4}\bigr)}
\frac{2^{s-2m}|z''|^{s-2m}}
{\bigl[(|z''|^{2}+|-z'+w'|^{2})^{2}+(-t+\mu-2\operatorname{Im} z' \cdot \overline{w}')^{2}\bigr]^{\frac{Q_n-4m+s}{4}}}\\
\label{3.5}
=& 2^{s-2m}\frac{\Gamma^{2}\bigl(\frac{Q_{n}+s-4m}{4}\bigr)}
{\Gamma^{2}\bigl(\frac{Q_{n}-s}{4}\bigr)}
|z''|^{s-2m}|u^{-1}v|^{4m-Q_n-s}.
\end{align}
Substituting (\ref{3.5}) into (\ref{4.2}) yields
\begin{align}\nonumber
&\frac{\pi^{m}\Gamma\bigl(\frac{s}{2}\bigr)}{\Gamma\bigl(\frac{s-2m}{2}\bigr)}
\mathcal{L}_{s,\mathbb{H}^{n}}^{-1/2}\mathcal{S}_{m}^{\ast}
\mathcal{L}_{s-2m,\mathbb{H}^{n-m}}^{1/2}\mathcal{R}_{m}f(u)\\
\nonumber
=& a_{n,s}\frac{\pi^{m}\Gamma\bigl(\frac{s}{2}\bigr)}{\Gamma\bigl(\frac{s-2m}{2}\bigr)}
2^{s-2m}\frac{\Gamma^{2}\bigl(\frac{Q_{n}+s-4m}{4}\bigr)}
{\Gamma^{2}\bigl(\frac{Q_{n}-s}{4}\bigr)}
|z''|^{s-2m}\int_{\mathbb{H}^{n-m}} |u^{-1}v|^{4m-Q_n-s}\mathcal{R}_{m}f(v)dv\\
\nonumber
=& \frac{2^{n-1-\frac{s}{2}}\Gamma^{2}\bigl(\frac{Q_n-s}{4}\bigr)}
{\pi^{n+1}\Gamma(\frac{s}{2})}
\frac{\pi^{m}\Gamma\bigl(\frac{s}{2}\bigr)}{\Gamma\bigl(\frac{s-2m}{2}\bigr)}
2^{s-2m}\frac{\Gamma^{2}\bigl(\frac{Q_{n}+s-4m}{4}\bigr)}
{\Gamma^{2}\bigl(\frac{Q_{n}-s}{4}\bigr)}
|z''|^{s-2m}
\int_{\mathbb{H}^{n-m}} |u^{-1}v|^{4m-Q_n-s}\mathcal{R}_{m}f(v)dv\\
=& \frac{2^{n+\frac{s}{2}-2m-1}\Gamma^{2}\bigl(\frac{Q_n+s-4m}{4}\bigr)}
{\pi^{n-m+1}\Gamma\bigl(\frac{s-2m}{2}\bigr)}
|z''|^{s-2m}
\int_{\mathbb{H}^{n-m}} |u^{-1}v|^{4m-Q_n-s}\mathcal{R}_{m}f(v)dv.
\label{c3.7}
\end{align}
This leads to the following lemma:
\begin{lemma}\label{lm4.2}
Let \(2m < s < Q_n\). For any \(f \in W^{s,2}(\mathbb{H}^{n})\), define
\[
\mathcal{P}_{s,m}(\mathcal{R}_{m}f)(u)=\frac{2^{n+\frac{s}{2}-2m-1}\Gamma^{2}\bigl(\frac{Q_n+s-4m}{4}\bigr)}
{\pi^{n-m+1}\Gamma\bigl(\frac{s-2m}{2}\bigr)}
|z''|^{s-2m}\int_{\mathbb{H}^{n-m}} |u^{-1}v|^{4m-Q_n-s}\mathcal{R}_{m}f(v)dv .
\]
Then
\begin{align}\nonumber
\int_{\mathbb{H}^{n}}\bigl|\mathcal{L}_{s,\mathbb{H}^{n}}^{1/2}f\bigr|^{2}
=& \int_{\mathbb{H}^{n}}\Bigl|\mathcal{L}_{s,\mathbb{H}^{n}}^{1/2}\bigl(f-
P_{s,m}(\mathcal{R}_{m}f)\bigr)\Bigr|^{2} \\
\label{b3.6}
&\qquad+ \frac{\pi^{m}\Gamma\bigl(\frac{s}{2}\bigr)}{\Gamma\bigl(\frac{s-2m}{2}\bigr)}
\int_{\mathbb{H}^{n-m}}\bigl|\mathcal{L}_{s-2m,\mathbb{H}^{n-m}}^{1/2}\mathcal{R}_{m}f\bigr|^{2}.
\end{align}

\end{lemma}
\begin{proof}
By Lemma \ref{lm4.1}, equality \eqref{b3.6} holds for every \(f \in C_{0}^{\infty}(\mathbb{H}^{n})\). Consequently,
\[
\frac{\pi^{m}\Gamma\bigl(\frac{s}{2}\bigr)}{\Gamma\bigl(\frac{s-2m}{2}\bigr)}
\int_{\mathbb{H}^{n-m}}\bigl|\mathcal{L}_{s-2m,\mathbb{H}^{n-m}}^{1/2}\mathcal{R}_{m}f\bigr|^{2}
\leq \int_{\mathbb{H}^{n}}\bigl|\mathcal{L}_{s,\mathbb{H}^{n}}^{1/2}f\bigr|^{2},
\qquad \forall f \in C_{0}^{\infty}(\mathbb{H}^{n}).
\]
Since \(C_{0}^{\infty}(\mathbb{H}^{n})\) is dense in \(W^{s,2}(\mathbb{H}^{n})\), the restriction operator \(\mathcal{R}_{m}\) extends to a bounded operator from \(W^{s,2}(\mathbb{H}^{n})\) to \(W^{s-2m,2}(\mathbb{H}^{n-m})\). By density, \eqref{b3.6} remains valid for all \(f \in W^{s,2}(\mathbb{H}^{n})\).
\end{proof}

\begin{remark}
It is straightforward to verify that \(\mathcal{P}_{s,m}(\mathcal{R}_{m}f)\) satisfies
\begin{align}\label{c3.8}
\mathcal{R}_{m}\mathcal{P}_{s,m}(\mathcal{R}_{m}f)=\mathcal{R}_{m}f,\quad f \in C_{0}^{\infty}(\mathbb{H}^{n}).
\end{align}
Indeed, using (\ref{c3.7}) and (\ref{3.4}) we obtain
\begin{align*}
\mathcal{L}_{s-2m,\mathbb{H}^{n-m}}^{1/2} \mathcal{R}_{m}\mathcal{P}_{s,m}(\mathcal{R}_{m}f)=&\frac{\pi^{m}\Gamma\bigl(\frac{s}{2}\bigr)}{\Gamma\bigl(\frac{s-2m}{2}\bigr)}
\mathcal{L}_{s-2m,\mathbb{H}^{n-m}}^{1/2} \mathcal{R}_{m}\mathcal{L}_{s,\mathbb{H}^{n}}^{-1/2}\mathcal{S}_{m}^{\ast}
\mathcal{L}_{s-2m,\mathbb{H}^{n-m}}^{1/2}\mathcal{R}_{m}f\\
=&\frac{\pi^{m}\Gamma\bigl(\frac{s}{2}\bigr)}{\Gamma\bigl(\frac{s-2m}{2}\bigr)}\mathcal{S}_{m}\mathcal{S}_{m}^{*}\mathcal{L}_{s-2m,\mathbb{H}^{n-m}}^{1/2}\mathcal{R}_{m}f\\
=&\mathcal{L}_{s-2m,\mathbb{H}^{n-m}}^{1/2}\mathcal{R}_{m}f,
\end{align*}
which immediately yields (\ref{c3.8}).
The same property can be checked in an analogous way for the operators \(\widetilde{\mathcal{P}}_{s,m}(\widetilde{\mathcal{R}}_{m}F)\), \(\mathcal{P}'_{s,m}(\mathcal{R}_{m}f)\),  \(\mathcal{Q}_{s,m}(\tau_{m}f)\) and $\widetilde{\mathcal{Q}}_{s,m}(\widetilde{\tau}_{m}F)$ that appear later in the paper.
\end{remark}

\textbf{Proof of Theorem \ref{th1.3}.} The theorem follows directly by combining the Sobolev inequality \eqref{Sobolev inequalities2} with Lemma \ref{lm4.2}.

\vspace{0.2cm}

We now proceed to the proof of Theorem \ref{th1.4}. To this end, we first establish the following lemma.

\begin{lemma}\label{lm4.3}
Let \(2m < s < Q_n\). For any \(F \in W^{s,2}(\mathbb{S}^{2n+1})\), define
\begin{align}\label{b3.2}
\widetilde{\mathcal{P}}_{s,m}(\widetilde{\mathcal{R}}_{m}F)
=
\frac{\Gamma^{2}\bigl(\frac{Q_n+s-4m}{4}\bigr)}
{2\pi^{n-m+1}\Gamma\bigl(\frac{s-2m}{2}\bigr)}
\int_{\mathbb{S}^{2(n-m)+1}} \frac{\bigl(\sum\limits_{j=n-m+2}^{n+1}|\zeta_{j}|^{2}\bigr)^{\frac{s-2m}{2}}}
{|1-\zeta\cdot\overline{\eta}|^{\frac{Q_n+s-4m}{2}}}
\widetilde{\mathcal{R}}_{m}F(\eta)d\eta .
\end{align}
Then
\begin{align}\nonumber
\int_{\mathbb{S}^{2n+1}}\bigl|\mathcal{A}_{s,\mathbb{S}^{2n+1}}^{1/2}F\bigr|^{2}
=& \int_{\mathbb{S}^{2n+1}}\Bigl|\mathcal{A}_{s,\mathbb{S}^{2n+1}}^{1/2}\bigl(F-
\widetilde{\mathcal{P}}_{s,m}(\widetilde{\mathcal{R}}_{m}F)\bigr)\Bigr|^{2} \\
\label{b4.3}
&\quad + \frac{\pi^{m}\Gamma\bigl(\frac{s}{2}\bigr)}{\Gamma\bigl(\frac{s-2m}{2}\bigr)}
\int_{\mathbb{S}^{2(n-m)+1}}\bigl|\mathcal{A}_{s-2m,\mathbb{S}^{2(n-m)+1}}^{1/2}
\widetilde{\mathcal{R}}_{m}F\bigr|^{2}.
\end{align}
Moreover, $\widetilde{\mathcal{P}}_{s,m}(\widetilde{\mathcal{R}}_{m}F)$ satisfies
\begin{align}\nonumber
\widetilde{\mathcal{P}}_{s,m}(\widetilde{\mathcal{R}}_{m}F)
= &\frac{2^{n-1-\frac{s}{2}}\Gamma^{2}\bigl(\frac{Q_n-s}{4}\bigr)}
{\pi^{n-m+1}\Gamma\bigl(\frac{s-2m}{2}\bigr)}\times\\
\label{4.3}
&
\int_{\mathbb{S}^{2(n-m)+1}}|1-\zeta\cdot\overline{\eta}|^{s-Q_n}
\mathcal{A}_{s-2m,\mathbb{S}^{2(n-m)+1}}\widetilde{\mathcal{R}}_{m}F(\eta)d\eta .
\end{align}
\end{lemma}

\begin{proof}
Without loss of generality, we may assume \(F \in C^{\infty}(\mathbb{S}^{2n+1})\).
Consider
\[
\widetilde{\mathbb{S}}^{2(n-m)+1}
= \bigl\{(\zeta_{1},\dots,\zeta_{n-m},0,\dots,0,\zeta_{n+1})\in \mathbb{S}^{2n+1} \bigr\},
\]
which satisfies
\[
\mathscr{C}_{n}(\mathbb{H}^{n-m}) = \widetilde{\mathbb{S}}^{2(n-m)+1}\setminus\{(0,\dots,0,1)\}.
\]
It suffices to verify the statement for \(\widetilde{\mathbb{S}}^{2(n-m)+1}\).
For simplicity, we continue to denote by \(\widetilde{\mathcal{R}}_{m}\) the restriction operator that maps  functions  on \(\mathbb{S}^{2n+1}\) to \(\widetilde{\mathbb{S}}^{2(n-m)+1}\).

Observe that for \(\zeta = \mathscr{C}_{n}(u) = \mathscr{C}_{n}(z,t)\) and \(\eta = \mathscr{C}_{n}(v) = \mathscr{C}_{n}(w,\mu)\) we have (see \cite{fl2})
\begin{align}\label{a4.3}
|1-\zeta \cdot \bar{\eta}|
= 2\bigl((1+|z|^2)^2 + t^2\bigr)^{-1/2}
|u^{-1}v|^2
\bigl((1+|w|^2)^2 + \mu^2\bigr)^{-1/2}.
\end{align}
Using (\ref{2.1}) and (\ref{a4.3}) we obtain
\begin{align}\nonumber
&\frac{\bigl(\sum\limits_{j=n-m+1}^{n}|\zeta_{j}|^{2}\bigr)^{\frac{s-2m}{2}}}
{|1-\zeta\cdot\overline{\eta}|^{\frac{Q_n+s-4m}{2}}}\\
\nonumber
=& |z''|^{s-2m}
\Bigl(\frac{4}{(1+|z|^{2})^{2}+t^2}\Bigr)^{\frac{s-2m}{2}}
\Bigl(\frac{2}{(1+|z|^{2})^{2}+t^2}\Bigr)^{-\frac{Q_n+s-4m}{4}}\times\\
\nonumber
&
\frac{1}{|u^{-1}v|^{Q_n+s-4m}}
\Bigl(\frac{2}{(1+|w'|^{2})^{2}+\mu^2}\Bigr)^{-\frac{Q_n+s-4m}{4}}\\
\nonumber
=& 2^{s-2m}
\Bigl(\frac{1}{(1+|z|^{2})^{2}+t^2}\Bigr)^{-\frac{Q_n-s}{4}}
\frac{|z''|^{s-2m}}{|u^{-1}v|^{Q_n+s-4m}}
\Bigl(\frac{4}{(1+|w'|^{2})^{2}+\mu^2}\Bigr)^{-\frac{Q_n+s-4m}{4}}\\
\label{b3.1}
=& 2^{\frac{Q_n+s}{2}-2m}
(2|J_{\mathscr{C}_n}|)^{-\frac{Q_n-s}{2Q_n}}
\frac{|z''|^{s-2m}}{|u^{-1}v|^{Q_n+s-4m}}
\Bigl(\frac{4}{(1+|w'|^{2})^{2}+\mu^2}\Bigr)^{-\frac{Q_n+s-4m}{4}}.
\end{align}
Consequently,
\begin{align}\nonumber
&(2|J_{\mathscr{C}_n}|)^{\frac{Q_n-s}{2Q_n}}
\widetilde{\mathcal{P}}_{s,m}(\widetilde{\mathcal{R}}_{m}F\circ\mathscr{C}_{n})\\
\nonumber
=& \frac{2^{n+\frac{s}{2}-2m}\Gamma^{2}\bigl(\frac{Q_n+s-4m}{4}\bigr)}
{\pi^{n-m+1}\Gamma\bigl(\frac{s-2m}{2}\bigr)}
\frac{1}{2}\int_{\mathbb{H}^{n-m}}
\frac{|z''|^{s-2m}}{|u^{-1}v|^{Q_n+s-4m}}
\widetilde{\mathcal{R}}_{m}F\circ\mathscr{C}_{n}\times
\Bigl(\frac{4}{(1+|w'|^{2})^{2}+\mu^2}\Bigr)^{\frac{Q_n-s}{4}} dv\\
\nonumber
=& \frac{2^{n+\frac{s}{2}-2m-1}\Gamma^{2}\bigl(\frac{Q_n+s-4m}{4}\bigr)}
{\pi^{n-m+1}\Gamma\bigl(\frac{s-2m}{2}\bigr)}
\int_{\mathbb{H}^{n-m}}
\frac{|z''|^{s-2m}}{|u^{-1}v|^{Q_n+s-4m}}
\widetilde{\mathcal{R}}_{m}\bigl((2|J_{\mathscr{C}_n}|)^{\frac{Q_n-s}{2Q_n}}F\circ\mathscr{C}\bigr)\\
\label{3.10}
=& \mathcal{P}_{s,m}\bigl(\widetilde{\mathcal{R}}_{m}(2|J_{\mathscr{C}_n}|)^{\frac{Q_n-s}{2Q_n}}F\circ\mathscr{C}\bigr).
\end{align}
Combining Lemma \ref{lm4.2}, identity (\ref{2.6}) and relation (\ref{3.10}) yields (\ref{b4.3}).

Finally, to verify (\ref{4.3}), one may use (\ref{3.10}), (\ref{2.6}) and (\ref{4.2}), from which the formula follows directly.
\end{proof}

\textbf{Proof of Theorem \ref{th1.4}.} The theorem is an immediate consequence of the Sobolev inequality \eqref{Sobolev inequalities1} together with Lemma \ref{lm4.3}.

\vspace{0.2cm}

Next we prove Theorem \ref{th1.5}. We first establish the following lemma.
\begin{lemma}\label{lm4.4}
Let \(2m < s < Q_n\). For any \(f \in W^{s,2}(\mathbb{H}^{n})\), define
\begin{align}\nonumber
\mathcal{P}'_{s,m}(\mathcal{R}_{m}f)(u)=&\frac{2^{s-2m-1}\Gamma^{2}\bigl(\frac{Q_n+s-4m}{4}\bigr)}
{\pi^{n-m+1}\Gamma\bigl(\frac{s-2m}{2}\bigr)}\times\\
\label{3.11}
&
|z''|^{s-2m}\int_{\mathbb{S}^{2(n-m)+1}} |u^{-1}v|^{4m-Q_n-s}\mathcal{R}_{m}f(v)dv .
\end{align}
Then
\begin{align}\nonumber
\int_{\mathbb{H}^{n}}\bigl|\mathcal{L}_{s,\mathbb{H}^{n}}^{1/2}f\bigr|^{2}
=& \int_{\mathbb{H}^{n}}\Bigl|\mathcal{L}_{s,\mathbb{H}^{n}}^{1/2}\bigl(f-
\mathcal{P}'_{s,m}(\mathcal{R}_{m}f)\bigr)\Bigr|^{2} \\
\label{3.6}
&\qquad + 2\frac{\pi^{m}\Gamma\bigl(\frac{s}{2}\bigr)}{\Gamma\bigl(\frac{s-2m}{2}\bigr)}
\int_{\mathbb{S}^{2(n-m)+1}}\bigl|\mathcal{A}_{s-2m,\mathbb{S}^{2(n-m)+1}}^{1/2}
\mathcal{R}_{m}f\bigr|^{2}.
\end{align}
\end{lemma}
\begin{proof}
Without loss of generality, assume \(f \in C_{0}^{\infty}(\mathbb{H}^{n})\).
Define
\[
F = \bigl( (2|J_{\mathscr{C}_{n}}|)^{-\frac{Q_n-s}{2Q_n}} f \bigr)\circ\mathscr{C}_{n}^{-1}.
\]
From (\ref{a2.6}) and (\ref{3.10}) we obtain
\[
\int_{\mathbb{S}^{2n+1}} \bigl|\mathcal{A}_{s,\mathbb{S}^{2n+1}}^{1/2} F\bigr|^{2}
= \frac{1}{2}\int_{\mathbb{H}^{n}} \bigl|\mathcal{L}_{s,\mathbb{H}^{n}}^{1/2} f\bigr|^{2}.
\]

Observe that the Cayley transform \(\mathscr{C}_{n}\) preserves the sphere
\[
\mathbb{S}^{2(n-m)+1} = \Bigl\{(\zeta_{1},\dots,\zeta_{n}, 0)\in \mathbb{H}^{n} :
\zeta_{j}=0 \text{ for } n-m+2 \leq j \leq n,\
\sum_{j=1}^{n-m+1}|\zeta_{j}|^{2}=1 \Bigr\},
\]
and that \(2|J_{\mathscr{C}_{n}}|(u)=1\) for every \(u \in \mathbb{S}^{2(n-m)+1}\). Consequently,
\[
\int_{\mathbb{S}^{2(n-m)+1}} \bigl|\mathcal{A}_{s-2m,\mathbb{S}^{2(n-m)+1}}^{1/2}
\widetilde{\mathcal{R}}_{m}F\bigr|^{2}
= \int_{\mathbb{S}^{2(n-m)+1}} \bigl|\mathcal{A}_{s-2m,\mathbb{S}^{2(n-m)+1}}^{1/2}
\mathcal{R}_{m}f\bigr|^{2},
\]
and
\begin{align*}
&\int_{\mathbb{S}^{2n+1}} \Bigl|\mathcal{A}_{s,\mathbb{S}^{2n+1}}^{1/2}\bigl(F-
\widetilde{\mathcal{P}}_{s,m}(\widetilde{\mathcal{R}}_{m}F)\bigr)\Bigr|^{2}\\
=& \frac{1}{2}
\int_{\mathbb{H}^{n}} \Bigl|\mathcal{L}_{s,\mathbb{H}^{n}}^{1/2}\bigl(f - (2|J_{\mathscr{C}_{n}}|)^{\frac{Q_n-s}{2Q_n}}\widetilde{\mathcal{P}}_{s,m}(\widetilde{\mathcal{R}}_{m}F)\circ\mathscr{C}_n\bigr)\Bigr|^{2}.
\end{align*}
Using (\ref{b3.2}) and (\ref{b3.1}) we obtain
\begin{align*}
&(2|J_{\mathscr{C}_{n}}|)^{\frac{Q_n-s}{2Q_n}}\widetilde{\mathcal{P}}_{s,m}(\widetilde{\mathcal{R}}_{m}F)\circ\mathscr{C}_n \\
=& 2^{s-2m}
\frac{\Gamma^{2}\bigl(\frac{Q_n+s-4m}{4}\bigr)}
{2\pi^{n-m+1}\Gamma\bigl(\frac{s-2m}{2}\bigr)}
|z''|^{s-2m}\int_{\mathbb{S}^{2(n-m)+1}} |u^{-1}v|^{4m-Q_n-s}\mathcal{R}_{m}f(v)dv \\
=& \mathcal{P}'_{s,m}(\mathcal{R}_{m}f).
\end{align*}
Thus
\begin{align*}
&\int_{\mathbb{S}^{2n+1}} \Bigl|\mathcal{A}_{s,\mathbb{S}^{2n+1}}^{1/2}\bigl(F-
\widetilde{\mathcal{P}}_{s,m}(\widetilde{\mathcal{R}}_{m}F)\bigr)\Bigr|^{2}= \frac{1}{2}
\int_{\mathbb{H}^{n}} \Bigl|\mathcal{L}_{s,\mathbb{H}^{n}}^{1/2}\bigl(f - \mathcal{P}'_{s,m}(\mathcal{R}_{m}f)\bigr)\Bigr|^{2}.
\end{align*}
Finally, combining the preceding identities with Lemma \ref{lm4.3} yields (\ref{3.6}).
\end{proof}

\textbf{Proof of Theorem \ref{th1.5}.} The result follows directly from the Sobolev inequality \eqref{Sobolev inequalities1} together with Lemma \ref{lm4.4}.

\vspace{0.2cm}

Finally, we consider the limiting case \( s = Q_n \). We begin with the following lemma.

\begin{lemma}\label{lm4.4}
If \( F \in C^{\infty}(\mathbb{S}^{2n+1}) \cap \mathcal{P}_n \), then
\[
\widetilde{\mathcal{P}}_{Q_n,m}(\widetilde{\mathcal{R}}_{m}F)
:= \lim_{s\rightarrow Q_n} \widetilde{\mathcal{P}}_{s,m}(\widetilde{\mathcal{R}}_{m}F) \; \in \mathcal{P}_n .
\]
\end{lemma}
\begin{proof}
From \eqref{4.3} we have
\begin{align*}
&\widetilde{\mathcal{P}}_{s,m}(\widetilde{\mathcal{R}}_{m}F)(\zeta) \\
=& \frac{2^{n-1-\frac{s}{2}}\Gamma^{2}\bigl(\frac{Q_n-s}{4}\bigr)}
{\pi^{n-m+1}\Gamma\bigl(\frac{s-2m}{2}\bigr)}
\int_{\mathbb{S}^{2(n-m)+1}}\mathcal{A}_{s-2m,\mathbb{S}^{2(n-m)+1}}
\widetilde{\mathcal{R}}_{m}F(\eta)d\eta \\
&\quad + \frac{2^{n-1-\frac{s}{2}}\Gamma^{2}\bigl(\frac{Q_n-s}{4}\bigr)}
{\pi^{n-m+1}\Gamma\bigl(\frac{s-2m}{2}\bigr)}
\int_{\mathbb{S}^{2(n-m)+1}}\bigl(|1-\zeta\cdot\overline{\eta}|^{s-Q_n}-1\bigr)
\mathcal{A}_{s-2m,\mathbb{S}^{2(n-m)+1}}\widetilde{\mathcal{R}}_{m}F(\eta)d\eta .
\end{align*}
Using
\[
\mathcal{A}_{s-2m,\mathbb{S}^{2(n-m)+1}}1
= \frac{\Gamma^{2}\bigl(\frac{Q_{n-m}+(s-2m)}{4}\bigr)}
{\Gamma^{2}\bigl(\frac{Q_{n-m}-(s-2m)}{4}\bigr)}
= \frac{\Gamma^{2}\bigl(\frac{Q_{n}+s-4m}{4}\bigr)}
{\Gamma^{2}\bigl(\frac{Q_{n}-s}{4}\bigr)},
\]
we obtain
\begin{align*}
&\widetilde{\mathcal{P}}_{s,m}(\widetilde{\mathcal{R}}_{m}F) \\
=& \frac{2^{n-1-\frac{s}{2}}\Gamma^{2}\bigl(\frac{Q_n+s-4m}{4}\bigr)}
{\pi^{n-m+1}\Gamma\bigl(\frac{s-2m}{2}\bigr)}
\int_{\mathbb{S}^{2(n-m)+1}}\widetilde{\mathcal{R}}_{m}F(\eta)d\eta \\
&\quad + \frac{2^{n-1-\frac{s}{2}}\Gamma^{2}\bigl(\frac{Q_n-s}{4}\bigr)}
{\pi^{n-m+1}\Gamma\bigl(\frac{s-2m}{2}\bigr)}
\int_{\mathbb{S}^{2(n-m)+1}}\bigl(|1-\zeta\cdot\overline{\eta}|^{s-Q_n}-1\bigr)
\mathcal{A}_{s-2m,\mathbb{S}^{2(n-m)+1}}\widetilde{\mathcal{R}}_{m}F(\eta)d\eta .
\end{align*}
Since \(\widetilde{\mathcal{R}}_{m}F(\eta) \in \mathcal{P}_{n-m}\) (because \(F \in \mathcal{P}_n\)), it follows that
\begin{align*}
&\widetilde{\mathcal{P}}_{Q_n,m}
= \lim_{s\rightarrow Q_n}\widetilde{\mathcal{P}}_{s,m}(\widetilde{\mathcal{R}}_{m}F) \\
=& \frac{2^{n-1-\frac{Q_n}{2}}\Gamma\bigl(\frac{Q_n-2m}{2}\bigr)}
{\pi^{n-m+1}} \int_{\mathbb{S}^{2(n-m)+1}} \widetilde{\mathcal{R}}_{m}F(\eta)d\eta \\
&\quad - \frac{2^{n+1-\frac{Q_n}{2}}}{\pi^{n-m+1}}
\int_{\mathbb{S}^{2(n-m)+1}} \ln|1-\zeta\cdot\overline{\eta}|
\mathcal{A}'_{Q_{n-m},\mathbb{S}^{2(n-m)+1}}\widetilde{\mathcal{R}}_{m}F(\eta)d\eta .
\end{align*}
As a function of the variable \(\zeta\), \(\ln|1-\zeta\cdot\overline{\eta}|\) belongs to \(\mathcal{P}_n\); consequently,
\(\widetilde{\mathcal{P}}_{Q_n,m}(\widetilde{\mathcal{R}}_{m}F)\) also lies in \(\mathcal{P}_n\).
\end{proof}

\begin{lemma}\label{lm4.5}
For any \(F \in W^{Q_n,2}(\mathbb{S}^{2n+1})\),
\begin{align}
\int_{\mathbb{S}^{2n+1}}\bigl|[\mathcal{A}'_{Q_n,\mathbb{S}^{2n+1}}]^{1/2}F\bigr|^{2}
=& \int_{\mathbb{S}^{2n+1}}\Bigl|[\mathcal{A}'_{Q_n,\mathbb{S}^{2n+1}}]^{1/2}\bigl(F-
\widetilde{\mathcal{P}}_{Q_n,m}(\widetilde{\mathcal{R}}_{m}F)\bigr)\Bigr|^{2} \nonumber\\
& + \pi^{m}\int_{\mathbb{S}^{2(n-m)+1}}
\bigl|(\mathcal{A}'_{Q_n-2m,\mathbb{S}^{2(n-m)+1}})^{1/2}\widetilde{\mathcal{R}}_{m}F\bigr|^{2}. \label{4.5}
\end{align}
\end{lemma}

\begin{proof}
Without loss of generality, assume \(F \in C^{\infty}(\mathbb{S}^{2n+1}) \cap \mathcal{P}_n\).
Using Lemma \ref{lm4.3}, Lemma \ref{lm4.4} and \eqref{1.7}, we obtain
\begin{align*}
&\int_{\mathbb{S}^{2n+1}}\bigl|[\mathcal{A}'_{Q_n,\mathbb{S}^{2n+1}}]^{1/2}F\bigr|^{2}\\
=& -\frac{4}{\Gamma(\frac{Q_n}{2})}
\lim_{s\rightarrow Q_n}\frac{1}{s-Q_n}
\int_{\mathbb{S}^{2n+1}}\bigl|\mathcal{A}_{s,\mathbb{S}^{2n+1}}^{1/2}F\bigr|^{2}\\
=& -\frac{4}{\Gamma(\frac{Q_n}{2})}
\lim_{s\rightarrow Q_n}\frac{1}{s-Q_n}
\int_{\mathbb{S}^{2n+1}}
\Bigl|\mathcal{A}_{s,\mathbb{S}^{2n+1}}^{1/2}\bigl(F-
\widetilde{\mathcal{P}}_{s,m}(\widetilde{\mathcal{R}}_{m}F)\bigr)\Bigr|^{2} \\
&\qquad - \frac{\pi^{m}\Gamma\bigl(\frac{Q_n}{2}\bigr)}{\Gamma\bigl(\frac{Q_n-2m}{2}\bigr)}
\frac{4}{\Gamma(\frac{Q_n}{2})}
\lim_{s\rightarrow Q_n}\frac{1}{s-Q_n}
\int_{\mathbb{S}^{2(n-m)+1}}\bigl|\mathcal{A}_{s-2m,\mathbb{S}^{2(n-m)+1}}^{1/2}
\widetilde{\mathcal{R}}_{m}F\bigr|^{2} \\
=& \int_{\mathbb{S}^{2n+1}}\Bigl|[\mathcal{A}'_{Q_n,\mathbb{S}^{2n+1}}]^{1/2}\bigl(F-
\widetilde{\mathcal{P}}_{Q_n,m}(\widetilde{\mathcal{R}}_{m}F)\bigr)\Bigr|^{2} \\
&\qquad + \pi^{m}\int_{\mathbb{S}^{2(n-m)+1}}
\bigl|(\mathcal{A}'_{Q_n-2m,\mathbb{S}^{2(n-m)+1}})^{1/2}\widetilde{\mathcal{R}}_{m}F\bigr|^{2}.
\end{align*}
\end{proof}

\textbf{Proof of Theorem \ref{th1.7}.} The theorem follows directly by combining Theorem \ref{brl} with Lemma \ref{lm4.5}.

\section{Proof of Theorem \ref{th1.8}}

Let $m < s < n$ and let
\[
\widetilde{\mathcal{S}}_{m}: C_{0}^{\infty}(\mathbb{R}^n)\subset L^2(\mathbb{R}^{n})\longrightarrow L^2(\mathbb{R}^{n-m})
\]
be the linear operator defined as
\[
\widetilde{\mathcal{S}}_{m}f = (-\Delta_{\mathbb{R}^{n-m}})^{\frac{s-m}{4}}
\tau_{m}(-\Delta_{\mathbb{R}^{n}})^{-\frac{s}{4}}f .
\]
Denote by \(\widetilde{\mathcal{S}}^{\ast}_{m}: L^2(\mathbb{R}^{n-m})\to L^2(\mathbb{R}^{n})\) the adjoint operator  of \(\widetilde{\mathcal{S}}_{m}\).
Using the well-known formula
\[
(-\Delta_{\mathbb{R}^n})^{-\alpha/2} \delta_{0}= \frac{1}{\gamma_{n}(\alpha)}
\frac{1}{|x|^{n-\alpha}}, \qquad
\gamma_{n}(\alpha) = \pi^{n/2}2^{\alpha}
\frac{\Gamma\bigl(\frac{\alpha}{2}\bigr)}{\Gamma\bigl(\frac{n-\alpha}{2}\bigr)},
\]
we obtain
\[
\widetilde{\mathcal{S}}^{\ast}_{m}g(x) = \frac{1}{\gamma_{n}(\frac{s}{2})}
\int_{\mathbb{R}^{n-m}} \frac{1}{|x-y'|^{n-\frac{s}{2}}}
(-\Delta_{\mathbb{R}^{n-m}})^{\frac{s-m}{4}}g(y')dy'.
\]
Consequently,
\begin{align}\nonumber
\widetilde{\mathcal{S}}_{m}\widetilde{\mathcal{S}}_{m}^{\ast}g
=& (-\Delta_{\mathbb{R}^{n-m}})^{\frac{s-m}{4}}
\frac{1}{\gamma_{n}(s)}\tau_{m}
\int_{\mathbb{R}^{n-m}} \frac{1}{|x-y'|^{n-s}}
(-\Delta_{\mathbb{R}^{n-m}})^{\frac{s-m}{4}}g(y')dy' \\
\nonumber
=& \frac{\gamma_{n-m}(s-m)}{\gamma_{n}(s)} (-\Delta_{\mathbb{R}^{n-m}})^{\frac{s-m}{4}}(-\Delta_{\mathbb{R}^{n-m}})^{\frac{m-s}{2}}
(-\Delta_{\mathbb{R}^{n-m}})^{\frac{s-m}{4}} g\\
\label{5.1}
=&\frac{\gamma_{n-m}(s-m)}{\gamma_{n}(s)}g.
\end{align}

\begin{lemma}\label{lm5.1}
For any $f\in C_{0}^{\infty}(\mathbb{R}^{n})$ and $0<s< n$, it holds
\begin{align*}
\int_{\mathbb{R}^{n}}|(-\Delta_{\mathbb{R}^{n}})^{\frac{s}{4}}f|^{2}=&\int_{\mathbb{R}^{n}}\left|(-\Delta_{\mathbb{R}^{n}})^{\frac{s}{4}}(f-
\frac{\gamma_{n}(s)}{\gamma_{n-m}(s-m)}(-\Delta_{\mathbb{R}^{n}})^{-\frac{s}{4}}\widetilde{\mathcal{S}}_m^{\ast}(-\Delta_{\mathbb{R}^{n-m}})^{\frac{s-m}{4}} \tau_{m}f)\right|^{2}\\
&+
\frac{\gamma_{n}(s)}{\gamma_{n-m}(s-m)}
\int_{\mathbb{R}^{n-m}}|(-\Delta_{\mathbb{R}^{n-m}})^{\frac{s-m}{4}}\tau_{m}f|^{2}.
\end{align*}
\end{lemma}
\begin{proof}
The proof is completely analogous to that of Lemma \ref{lm4.1} and is therefore omitted.
\end{proof}

Next we compute the expression \( (-\Delta_{\mathbb{R}^{n}})^{-\frac{s}{4}}\widetilde{\mathcal{S}}_{m}^{\ast}(-\Delta_{\mathbb{R}^{n-m}})^{\frac{s-m}{4}} \tau_{m}f \). Observe that
\begin{align}\nonumber
&(-\Delta_{\mathbb{R}^{n}})^{-\frac{s}{4}}\widetilde{\mathcal{S}}_{m}^{\ast}(-\Delta_{\mathbb{R}^{n-m}})^{\frac{s-m}{4}} \tau_{m}f\\
\nonumber
=&(-\Delta_{\mathbb{R}^{n}})^{-\frac{s}{4}}
\frac{1}{\gamma_{n}(\frac{s}{2})}\int_{\mathbb{R}^{n-m}}\frac{1}{|x-y'|^{n-\frac{s}{2}}}
(-\Delta_{ \mathbb{R}^{n-m}})^{\frac{s-m}{2}}\tau_{m}f(y')dy'\\[2pt]
\label{5.2}
=&\frac{1}{\gamma_{n}(s)}
\int_{\mathbb{R}^{n-m}}\frac{1}{|x-y'|^{n-s}}
(-\Delta_{ \mathbb{R}^{n-m}})^{\frac{s-m}{2}}\tau_{m}f(y')dy'.
\end{align}
Write \( x=(x',x'') \) and \( y=(y',y'') \) with \( y''=0 \). Then
\begin{align}\label{b5.1}
\frac{1}{|x-y'|^{n-s}}
= \bigl(|x''|^{2}+|x'-y'|^{2}\bigr)^{\frac{s-n}{2}}
= |x''|^{s-n}
\Bigl(\frac{1}{\frac{|x'-y'|^{2}}{|x''|^{2}}+1}\Bigr)^{\frac{n-m-(s-m)}{2}}.
\end{align}

Consider the stereographic projection \(\mathscr{S}_{n}: \mathbb{R}^n \to \mathbb{S}^n\) and its inverse \(\mathscr{S}_{n}^{-1}: \mathbb{S}^n \to \mathbb{R}^n\), given by
\[
\mathscr{S}_{n}(x) = \Bigl( \frac{2x}{1+|x|^2}, \frac{1-|x|^2}{1+|x|^2} \Bigr),\qquad
\mathscr{S}_{n}^{-1}(\omega) = \Bigl( \frac{\omega_1}{1+\omega_{n+1}},\dots, \frac{\omega_n}{1+\omega_{n+1}} \Bigr).
\]
Its Jacobian is
\begin{align}\label{Jacobian}
|J_{\mathscr{S}_{n}}|(x) = \Bigl( \frac{2}{1+|x|^2} \Bigr)^n.
\end{align}
The operators \(P_{s,\mathbb{S}^{n}}\) and \(\Delta_{\mathbb{R}^{n}}\) are related by the following intertwining formula (see \cite{Morpurgo}): for \(F\in C^{\infty}(\mathbb{S}^{n})\),
\begin{align}\label{5.6}
(P_{s,\mathbb{S}^{n}} F) \circ \mathscr{S}_{n}
= |J_{\mathscr{S}_{n}}|^{-(n+s)/(2n)}
(-\Delta_{\mathbb{R}^{n}})^{s/2}\bigl( |J_{\mathscr{S}}|^{(n-s)/(2n)} (F \circ \mathscr{S}_{n}) \bigr).
\end{align}
Choosing \(F\equiv 1\) gives
\begin{align}\label{b5.7}
(-\Delta_{\mathbb{R}^{n}})^{s/2}|J_{\mathscr{S}_{n}}|^{(n-s)/(2n)}
= \frac{\Gamma\bigl(\frac{n+s}{2}\bigr)}{\Gamma\bigl(\frac{n-s}{2}\bigr)}
|J_{\mathscr{S}_{n}}|^{(n+s)/(2n)} .
\end{align}
Combing (\ref{b5.1}) and (\ref{b5.7}) yields
\begin{align}\nonumber
&(-\Delta_{ \mathbb{R}^{n-m}})^{\frac{s-m}{2}}\frac{1}{|x-y'|^{n-s}}\\
\nonumber
=& (-\Delta_{ \mathbb{R}^{n-m}})^{\frac{s-m}{2}}
|x''|^{s-n}
\Bigl(\frac{1}{\frac{|x'-y'|^{2}}{|x''|^{2}}+1}\Bigr)^{\frac{n-m-(s-m)}{2}}\\
\nonumber
=& 2^{s-m}\frac{\Gamma\bigl(\frac{n-m+s-m}{2}\bigr)}
{\Gamma\bigl(\frac{n-m-s+m}{2}\bigr)}
|x''|^{s-n}|x''|^{m-s}
\Bigl(\frac{1}{\frac{|x'-y'|^{2}}{|x''|^{2}}+1}\Bigr)^{\frac{n-m+(s-m)}{2}}\\[2pt]
\label{5.3}
=& 2^{s-m}\frac{\Gamma\bigl(\frac{n+s-2m}{2}\bigr)}
{\Gamma\bigl(\frac{n-s}{2}\bigr)}
|x''|^{s-m}
\frac{1}{|x-y'|^{n+s-2m}} .
\end{align}
Inserting (\ref{5.3}) into (\ref{5.2}) yields
\begin{align}\nonumber
&(-\Delta_{\mathbb{R}^{n}})^{-\frac{s}{4}}\mathcal{S}_{m}^{\ast}
(-\Delta_{\mathbb{R}^{n-m}})^{\frac{s-m}{4}} \tau_{m}f\\[2pt]
\nonumber
=& \frac{1}{\gamma_{n}(s)}
\int_{\mathbb{R}^{n-m}}
(-\Delta_{ \mathbb{R}^{n-m}})^{\frac{s-m}{2}}\frac{1}{|x-y'|^{n-s}}
\tau_{m}f(y')dy'\\[2pt]
\nonumber
=& \frac{\Gamma\bigl(\frac{n-s}{2}\bigr)}
{\pi^{n/2}2^{s}\Gamma(\frac{s}{2})} 2^{s-m}\frac{\Gamma\bigl(\frac{n+s-2m}{2}\bigr)}
{\Gamma\bigl(\frac{n-s}{2}\bigr)}
\int_{\mathbb{R}^{n-m}}\frac{|x''|^{s-m}}{|x-y'|^{n+s-2m}}
\tau_{m}f(y')dy'\\[2pt]
\label{5.4}
=& \frac{\Gamma\bigl(\frac{n+s-2m}{2}\bigr)}
{\pi^{n/2}2^{m}\Gamma(\frac{s}{2})}
\int_{\mathbb{R}^{n-m}}\frac{|x''|^{s-m}}{|x-y'|^{n+s-2m}}
\tau_{m}f(y')dy'.
\end{align}
Therefore, if we define
\[
\mathcal{Q}_{s,m}(\tau_{m}f)(x)=\frac{\gamma_{n}(s)}{\gamma_{n-m}(s-m)}
\frac{\Gamma\bigl(\frac{n+s-2m}{2}\bigr)}{\pi^{n/2}2^{m}\Gamma(\frac{s}{2})}
\int_{\mathbb{R}^{n-m}}\frac{|x''|^{s-m}}{|x-y'|^{n+s-2m}}
\tau_{m}f(y')dy',
\]
then by Lemma \ref{lm5.1}, we obtain the following lemma:
\begin{lemma}\label{lm5.2}
Let \(m<s< n\). For any \(f\in D_{s}(\mathbb{R}^{n})\),
\begin{align*}
\int_{\mathbb{R}^{n}}\bigl|(-\Delta_{\mathbb{R}^{n}})^{\frac{s}{4}}f\bigr|^{2}
=& \int_{\mathbb{R}^{n}}\Bigl|(-\Delta_{\mathbb{R}^{n}})^{\frac{s}{4}}\bigl(f-
\mathcal{Q}_{s,m}(\tau_{m}f)\bigr)\Bigr|^{2}
+ \frac{\gamma_{n}(s)}{\gamma_{n-m}(s-m)}
\int_{\mathbb{R}^{n-m}}\bigl|(-\Delta_{\mathbb{R}^{n-m}})^{\frac{s-m}{4}}\tau_{m}f\bigr|^{2}.
\end{align*}
\end{lemma}

We now extend Lemma \ref{lm5.2} to the sphere \(\mathbb{S}^{n}\).
For \(\omega = \mathscr{S}_{n}(x)\) and \(\eta = \mathscr{S}_{n}(y)\), it holds (see \cite{lieb2})
\begin{align}\label{5.5}
|\omega - \eta|^2 = \left( \frac{2}{1+|x|^2} \right) |x - y|^2 \left( \frac{2}{1+|y|^2} \right).
\end{align}
Consider
\[
\widetilde{\mathbb{S}}^{n-m}
= \bigl\{(\zeta_{1},\dots,\zeta_{n-m},0,\dots,0,\zeta_{n+1})\in \mathbb{S}^{n} \bigr\},
\]
which satisfies
\[
\mathscr{S}_{n}(\mathbb{R}^{n-m}) = \widetilde{\mathbb{S}}^{n-m}\setminus\{(0,\dots,0,1)\}.
\]
For simplicity, we continue to denote by \(\widetilde{\tau}_{m}\) the restriction operator that maps  functions  on \(\mathbb{S}^{n}\) to \(\widetilde{\mathbb{S}}^{n-m}\).

 Using Lemma \ref{lm5.2} together with the intertwining relation (\ref{5.6}), we obtain, for $F\in C^{\infty}(\mathbb{S}^{n})$,
\begin{align*}
&\int_{\mathbb{R}^{n}}\Bigl|(-\Delta_{\mathbb{R}^{n}})^{\frac{s}{4}}\bigl( |J_{\mathscr{S}_{n}}|^{\frac{n-s}{2n}} (F \circ \mathscr{S}_{n}) \bigr)\Bigr|^{2}\\
=& \int_{\mathbb{R}^{n}}\Bigl|(-\Delta_{\mathbb{R}^{n}})^{\frac{s}{4}}\Bigl( |J_{\mathscr{S}_{n}}|^{\frac{n-s}{2n}} (F \circ \mathscr{S}_{n}) -
\mathcal{Q}_{s,m}\bigl(\tau_{m}\bigl( |J_{\mathscr{S}_{n}}|^{\frac{n-s}{2n}} (F \circ \mathscr{S}_{n}) \bigr)\bigr)\Bigr)\Bigr|^{2}\\
&\quad + \frac{\gamma_{n}(s)}{\gamma_{n-m}(s-m)}
\int_{\mathbb{R}^{n-m}}\Bigl|(-\Delta_{\mathbb{R}^{n-m}})^{\frac{s-m}{4}}
\tau_{m}\bigl( |J_{\mathscr{S}_{n}}|^{\frac{n-s}{2n}} (F \circ \mathscr{S}_{n}) \bigr)\Bigr|^{2}.
\end{align*}

Observe that, by (\ref{5.5})  and (\ref{Jacobian}),
\begin{align*}
& |J_{\mathscr{S}_{n}}|^{-\frac{n-s}{2n}}
\mathcal{Q}_{s,m}\bigl(\tau_{m}\bigl( |J_{\mathscr{S}_{n}}|^{\frac{n-s}{2n}} (F \circ \mathscr{S}_{n}) \bigr)\bigr)\\
=& \frac{\gamma_{n}(s)}{\gamma_{n-m}(s-m)}
\frac{\Gamma\bigl(\frac{n+s-2m}{2}\bigr)}{\pi^{n/2}2^{m}\Gamma(\frac{s}{2})}
|J_{\mathscr{S}_{n}}|^{-\frac{n-s}{2n}}
|x''|^{s-m}
\int_{\mathbb{R}^{n-m}}\frac{1}{|x-y'|^{n+s-2m}}
\tau_{m}\bigl(|J_{\mathscr{S}_{n}}|^{\frac{n-s}{2n}} (F \circ \mathscr{S}_{n})\bigr)dy'\\
=& \frac{\gamma_{n}(s)}{\gamma_{n-m}(s-m)}
\frac{\Gamma\bigl(\frac{n+s-2m}{2}\bigr)}{\pi^{n/2}2^{m}\Gamma(\frac{s}{2})}
|J_{\mathscr{S}_{n}}|^{-\frac{n-s}{2n}}
|x''|^{s-m}
\int_{\mathbb{R}^{n-m}}\frac{1}{|x-y'|^{n+s-2m}}
|J_{\mathscr{S}_{n-m}}|^{\frac{n-s}{2(n-m)}} \tau_{m}(F \circ \mathscr{S}_{n})dy'\\[2pt]
=& \frac{\gamma_{n}(s)}{\gamma_{n-m}(s-m)}
\frac{\Gamma\bigl(\frac{n+s-2m}{2}\bigr)}{\pi^{n/2}2^{m}\Gamma(\frac{s}{2})}
\Bigl(\sum_{j=n-m+1}^{n}|\zeta_{j}|^{2}\Bigr)^{\frac{s-m}{2}}
\int_{\widetilde{\mathbb{S}}^{n-m}} \frac{1}{|\zeta-\overline{\eta}|^{n+s-2m}}
\widetilde{\tau}_{m}F(\eta)d\eta .
\end{align*}
Hence, if we set
\begin{align*}
\widetilde{\mathcal{Q}}_{s,m}(\widetilde{\tau}_{m}F)
=& \frac{\gamma_{n}(s)}{\gamma_{n-m}(s-m)}
\frac{\Gamma\bigl(\frac{n+s-2m}{2}\bigr)}{\pi^{n/2}2^{m}\Gamma(\frac{s}{2})}
\Bigl(\sum_{j=n-m+2}^{n+1}|\zeta_{j}|^{2}\Bigr)^{\frac{s-m}{2}}
\int_{\mathbb{S}^{n-m}} \frac{1}{|\zeta-\eta|^{n+s-2m}}
\widetilde{\tau}_{m}F(\eta)d\eta\\
=&\pi^{\frac{m-n}{2}}\frac{\Gamma\bigl(\frac{n+s-2m}{2}\bigr)}{\Gamma(\frac{s-m}{2})}\Bigl(\sum_{j=n-m+2}^{n+1}|\zeta_{j}|^{2}\Bigr)^{\frac{s-m}{2}}
\int_{\mathbb{S}^{n-m}} \frac{1}{|\zeta-\eta|^{n+s-2m}}
\widetilde{\tau}_{m}F(\eta)d\eta,
\end{align*}
then
we obtain the following lemma. (The case \(s=n\) is obtained by taking the limit \(s\to n\).)

\begin{lemma}\label{lm5.3}
Let \(m< s \leq n\).
For any \(F \in H^{s}(\mathbb{S}^{n})\),
\[
\int_{\mathbb{S}^{n}}\bigl|P_{s,\mathbb{S}^{n}}^{1/2}F\bigr|^{2}
= \int_{\mathbb{S}^{n}}\Bigl|P_{s,\mathbb{S}^{n}}^{1/2}\bigl(F-
\widetilde{\mathcal{Q}}_{s,m}(\widetilde{\tau}_{m}F)\bigr)\Bigr|^{2}
+ \pi^{m/2}  2^{m}  \frac{\Gamma\left(\frac{s}{2}\right)}{\Gamma\left(\frac{s-m}{2}\right)}
\int_{\mathbb{S}^{n-m}}\bigl|P_{s-m,\mathbb{S}^{n-m}}^{1/2}
\widetilde{\tau}_{m}F\bigr|^{2}.
\]
\end{lemma}

\textbf{Proof of Theorem \ref{th1.8}.} The theorem follows directly by combining Theorem \ref{Beckner} with Lemma \ref{lm5.3}.


\begin{thebibliography}{99}

\bibitem{Ache&Chang}
A.~G. Ache,  S.-Y.~A. Chang,
 Sobolev trace inequalities of order four,
 Duke Math. J. 166(14)(2017), 2719-2748.


\bibitem{be1} W. Beckner, Sharp Sobolev inequalities on the sphere and the Moser-Trudinger inequality,  Ann.  Math. (2) 138 (1993), 213-242.

\bibitem{be2} W. Beckner, Functionals for multilinear fractional embedding, Acta Math. Sin. 31 (2015), 1-28.



\bibitem{bez}N. Bez, S. Machihara, M. Sugimoto, Extremisers for the trace theorem on the sphere, Math. Res. Lett.
 23(2016), No.3, 633-647.






\bibitem{br1}T. P. Branson, L. Fontana,  C. Morpurgo, Moser-Trudinger and Beckner-Onofri's inequalities on the CR sphere, Ann. of  Math. (2), 177 (2013), 1-52.



\bibitem{ca}L. A. Caffarelli,  L. Silvestre, An extension problem related to the fractional Laplacian, Comm. Part. Diff. Equa., 32 (2007), 1245-1260.

\bibitem{Carlen} E. A. Carlen, M. Loss,  Competing symmetries of some functionals arising in
mathematical physics,  Stochastic processes, physics and geometry (Ascona and Locarno,
1988), 277-288, World Sci. Publ., Teaneck, NJ, 1990.

	\bibitem{Case0} J.~S. Case,
			Boundary operators associated with the Paneitz operator.
			Indiana Univ. Math. J. 67(2018), no. 1, 293-327.

\bibitem{Case1}
J.~S. Case,
 Some energy inequalities involving fractional GJMS operators.
 Anal. PDE, 10(2)( 2017), 253-280.

\bibitem{Case2}
J.~S. Case,
 Sharp weighted Sobolev trace inequalities and fractional powers of
  the Laplacian,
 J. Funct. Anal. 279(4)(2020), 108567.

\bibitem{Case&Chang}
J.~S. Case S.-Y.~A. Chang,
 On fractional {GJMS} operators.
 Comm. Pure Appl. Math. 69(6)(2016), 1017-1061.

 \bibitem{Case&Luo}
J.~S. Case,  W.~Luo,
 Boundary operators associated with the sixth-order GJMS operator,
 Int. Math. Res. Not. IMRN, 14(2021), 10600-10653.

 \bibitem{Cha}M. Chatzakou, A. Kassymov, M. Ruzhansky,  Logarithmic Sobolev, Hardy and Poincar\'e inequalities on
the Heisenberg group, arXiv:2310.00992

 \bibitem{Chen&Zhang} X. Chen, S. Zhang,  A sharp Sobolev trace inequality of order four on three-balls,  Journal d'Analyse Math\'ematique, accepted, see avalible  arXiv:2403.00380.

 \bibitem{Co} M. Cowling, Unitary and uniformly bounded representations of some simple Lie groups,
In: Analysis and Group Representations, C.I.M.E. Napoli; Liguori, 1982, 49-128.

 \bibitem{ei}A. Einav, M. Loss, Sharp trace inequalities for fractional Laplacians, Proc. Amer. Math. Soc.  140(2012), No.12, 4209-4216.

\bibitem{es} J. F.  Escobar,  Sharp constant in a Sobolev trace inequality,  Indiana Univ. Math. J. 37(1988), no. 3, 687-698.

 \bibitem{Flynn&Lu&Yang1}J. Flynn, G. Lu and Q. Yang, Conformally covariant boundary operators and sharp higher order CR Sobolev trace inequalities on the Siegel domain and complex ball, Journal f\"ur die reine und angewandte Mathematik (Crelles Journal), 824 (2025), 39-87.

\bibitem{Flynn&Lu&Yang}J. Flynn, G. Lu and Q. Yang, Conformally covariant boundary operators and sharp higher order Sobolev trace inequalities on Poincar\'e Einstein manifolds, arXiv:2311.10070.



\bibitem{fo}  G.B. Folland, Harmonic Analysis in Phase Space, Ann. of Math. Stud., vol. 122, Princeton University Press, Princeton, NJ, 1989.










\bibitem{fo}G. B. Folland, Spherical harmonic expansion of the Poisson-Szeg\"o kernel for the ball,
Proc. Amer. Math. Soc 47 (1975), No.2, 401-408.


\bibitem{fr1} R.L. Frank, M.d.M. Gonz\'alez, D. D. Monticelli, J. Tan, An extension problem for the CR fractional
Laplacian, Adv. Math. 270 (2015), 97-137.

\bibitem{fl2} R. L. Frank, E. Lieb, Sharp constants in several inequalities on the Heisenberg group, Ann.
 Math. (2) 176 (2012), No.1, 349-381.



\bibitem{Gho}S. Ghosh1,  V. Kumar, M. Ruzhansky, Best constants in subelliptic fractional Sobolev and
Gagliardo-Nirenberg inequalities and ground states on
stratified Lie groups, Calc. Var. (2026) 65:28.

\bibitem{Gong} R. Gong, Q. Yang, S. Zhang, A simple proof of reverse Sobolev inequalities on the sphere and  Sobolev trace inequalities on the unit ball, J. Func. Anal.,
290 (2026), no. 9, 111380.










\bibitem{lieb2}E. H. Lieb,  M. Loss, Analysis,  Second Edition.   Graduate Studies in Mathematics, Vol. 14.
American Mathematical Society, Providence, RI, 2001.



\bibitem{Morpurgo} C. Morpurgo, Sharp inequalities for functional integrals and traces of conformally invariant operators, Duke
Math. J. 114 (3) (2002),  477-553.



\bibitem{on}E. Onofri, On the positivity of the effective action in a theory of random surfaces,
Comm. Math. Phys. 86 (1982), 321-326.

\bibitem{ron1}L. Roncal, S. Thangavelu, Hardy's inequality for fractional powers of the sublaplacian on the Heisenberg group, Adv. Math. 302 (2016), 106-158.



\bibitem{st}   N.K. Stanton, Spectral invariants of CR manifolds, Michigan Math. J. 36 (1989), 267-288.



\bibitem{tha2} S. Thangavelu, An introduction to the uncertainty principle. Hardy's theorem on Lie groups,  Progress in Mathematics 217. Birkh\"auser, Boston, MA, 2004.

\bibitem{Wang} Y. Wang, Q. Yang,  Sharp fractional Sobolev and related inequalities on H-type groups,  arXiv:2406.16278v2

\bibitem{yang} Q. Yang,  Sharp Sobolev trace inequalities for higher order derivatives, arXiv:1901.03945.



\end{thebibliography}
\end{document}